\documentclass[12pt, reqno]{amsart}
\usepackage{amsmath, amsthm, amscd, amsfonts, amssymb, graphicx, color, mathrsfs}
\usepackage[bookmarksnumbered, colorlinks, plainpages]{hyperref}
\usepackage{cite}
\usepackage[all]{xy}
\usepackage{slashed}
\usepackage{enumitem}

\DeclareMathOperator{\Ad}{Ad}

\allowdisplaybreaks

\newlength\tindent
\newtheorem{theorem}{Theorem}[section]
\newtheorem{lemma}[theorem]{Lemma}

\newtheorem{proposition}[theorem]{Proposition}
\newtheorem{corollary}[theorem]{Corollary}

\theoremstyle{definition}
\newtheorem{definition}[theorem]{Definition}
\newtheorem{example}[theorem]{Example}

\theoremstyle{remark}
\newtheorem{remark}[theorem]{Remark}
\numberwithin{equation}{section}

\begin{document}
\setcounter{page}{1}

\title[Equigeodesics vectors on homogeneous spaces]{Equigeodesic vectors on compact homogeneous spaces with equivalent isotropy summands}

\author[B. Grajales]{Brian Grajales}
\address{Brian Grajales \endgraf
IMECC-Unicamp, Departamento de Matemática, Rua Sérgio Buarque de Holanda \endgraf
651, Cidade Universitária Zeferino Vaz. 13083-859, Campinas - SP, Brazil
\endgraf
  {\it E-mail address} {\rm grajales@ime.unicamp.br}
  }

\author[L. Grama]{Lino Grama}
\address{Lino Grama \endgraf
IMECC-Unicamp, Departamento de Matem\'{a}tica. Rua S\'{e}rgio Buarque de Holanda \endgraf 651, Cidade Universit\'{a}ria Zeferino Vaz. 13083-859, Campinas - SP, Brazil.
\endgraf
{\it E-mail address} {\rm lino@ime.unicamp.br}
  }


     \keywords{Equigeodesic, Invariant Geometry, Homogeneous space}
     \subjclass[2020]{53C30, 53C22}

\begin{abstract} In this paper, we investigate equigeodesics on a compact homogeneous space $M=G/H.$ We introduce a formula for the identification of equigeodesic vectors only relying on the isotropy representation of $M$ and the Lie structure of the Lie algebra of $G$. Applications to $M$-spaces are also discussed.
\end{abstract} 

\maketitle

\tableofcontents
\allowdisplaybreaks

\section{Introduction}

Homogeneous spaces have long been a captivating realm within the field of pure mathematics, offering a rich terrain for exploring the interplay between geometry, topology, and group theory. These spaces, equipped with natural group actions, have attracted the attention of mathematicians for their inherent symmetry and elegant mathematical structure. 
The study of differential geometry on homogeneous spaces is a classic topic in Differential Geometry with recent developments. See for instance \cite{agrico,ABD,AD,conti,FGV,FMR,FT,FTo,GGN,sn} and references therein.\\

Among the myriad of questions that arise in this context, the study of geodesics stands out as a central and challenging problem. Geodesics, as the shortest paths connecting points on a manifold, have played a fundamental role in the development of modern differential geometry. Understanding geodesics on homogeneous spaces not only extends our knowledge of geometric structures but also provides insights into various applied fields, including physics, computer graphics, and robotics. The intrinsic beauty and significance of this topic have motivated mathematicians to delve deep into the intricate world of geodesics on homogeneous spaces.\\

In this work, we are interested in investigating a specific class of geodesics on homogeneous spaces known as {\it homogeneous equigeodesics}. A homogeneous equigeodesic on a homogeneous space $G/H$ is a homogeneous curve $\gamma$ through the point $o=eH$, meaning $$\gamma(t) = \exp(tX)\cdot o,$$ such that $\gamma$ is a geodesic with respect to every invariant metric on $G/H$. The vector $ X\in T_o(G/H)$ is referred to as the {\it equigeodesic vector}. The notion of equigeodesic is important for mechanics since homogeneous geodesics are described relative equilibrium of the dynamical system, represented by an invariant metric and we consider the case when this property is stable under change of the metric.\\

The concept of homogeneous equigeodesics was initially introduced in \cite{CGN}, where the authors explored them within the context of generalized flag manifolds. They demonstrated the existence of these equigeodesics within generalized flag manifolds of type $A_l$. Several authors have contributed to the realm of equigeodesics within flag manifolds. For instance, in \cite{GN}, and \cite{X}, the authors focus on the study of equigeodesics on generalized flag manifolds with two, and four isotropy summands, respectively, in \cite{S2} are investigated equigeodesics on flag manifolds with $G_2$-type $\mathfrak{t}$-roots, and, in \cite{WZ}, the authors examine the existence and properties of equigeodesics in flag manifolds where the second Betti number $b_2(G/K) = 1.$ For other homogeneous spaces, it is noteworthy to mention the works of Statha \cite{S}, which includes a characterization of algebraic equigeodesics on some homogeneous spaces, such as Stiefel manifolds, generalized Wallach spaces, and some spheres, and Xu and Tan \cite{XT}, who have extended the concept of homogeneous equigeodesics to the context of homogeneous Finsler spaces, expanding the scope of this field of study.\\

In this paper, we embark on a comprehensive exploration of equigeodesic vectors within the general context of a compact homogeneous space. In Section \ref{Preliminaries}, we carefully detail the description of invariant metrics on homogeneous spaces. While such a description has appeared in the literature in various forms, we take a distinctive approach by providing a detailed exposition. This effort is not merely an exercise in redundancy but serves to establish a robust and consistent notation framework that supports our subsequent analyses. It is worth to point out that our framework is highly versatile, encompassing, for instance, homogeneous spaces where the isotropy representation has equivalent components. This general framework unifies terminology and extends the scope of several prior works.\\

One of the key achievements of this paper is the introduction of a formula that facilitates the identification of equigeodesic vectors without relying on the Riemannian invariant metrics (Theorem \ref{main:2}). This innovation holds immense importance as it simplifies the process of identifying equigeodesic vectors, making it more efficient for researchers in this field. This achievement is obtained through a careful analysis of both the isotropy representation and the underlying Lie theoretical structure within the homogeneous space.\\

In Section \ref{section:3}, we present a comprehensive overview of the general and established theory of homogeneous geodesics within homogeneous spaces and prove our main result (Theorem \ref{main:2}): a necessary condition for a vector to be equigeodesic. This condition provides an effective tool for finding equigeodesics on a compact homogeneous space. Furthermore, as demonstrated in Theorem \ref{main:3}, it serves as a definitive characterization for those homogeneous spaces where each isotropy summand with multiplicity greater than one is orthogonal. Finally, Section \ref{applications} is dedicated to the study of equigeodesics on $M$-spaces. We apply the results obtained in Section \ref{section:3} to give a description of equigeodesic vectors for these spaces.

\section{Invariant Metrics on Reductive Homogeneous Spaces}\label{Preliminaries}
Let $M$ be a differentiable manifold and $G$ a Lie group. We say that $M$ is a homogeneous $G$-space (or simply, homogeneous space)  if a smooth transitive action $\phi: G\times M\rightarrow M$ exists. Given $a\in G$ and $p\in M$, we will denote $a\cdot p=\phi(a,p)$ and define the maps \begin{center}
$\begin{array}{rccc}
\phi_a: & M & \longrightarrow & M\\
& x & \longmapsto 
 & a\cdot x
\end{array}$\  and \  $\begin{array}{rccc}
\varphi_p: & G & \longrightarrow & M\\
& b & \longmapsto 
 & b\cdot p
\end{array}.$
\end{center}

\begin{proposition}\label{diff:submersion} For every $a\in G$ and $p\in M$, $\phi_a$ is a diffeomorphism and $\varphi_p$ is a surjective submersion. 
\end{proposition}
\begin{proof} Since the action $\phi$ is smooth, then $\phi_a$ and $\varphi_p$ are smooth for each $a\in G$ and $p\in M$. Now, observe that $\phi_a\circ\phi_{a^{-1}}=\phi_{a^{-1}}\circ\phi_{a}=\textnormal{I}_M$, so $\phi_a$ is a diffeomorphism. Before proving that $\varphi_p$ is a surjective submersion, let us define the isotropy subgroup of an element $p\in M$ as $G_p:=\{a\in G:a\cdot p=p\}$ and denote by $L_a:G\rightarrow G;\ b\mapsto ab$ the left translation with respect to $a.$ Observe that $L_a$ is a diffeomorphism of $G$ and that for every $a\in G$ we have
\begin{align*}
    [\phi_{a^{-1}}\circ\varphi_p\circ L_{a}](b)=&[\phi_{a^{-1}}\circ\varphi_p](ab)\\
    =&\phi_{a^{-1}}((ab)\cdot p)\\
    =&a^{-1}\cdot[(ab)\cdot p]\\
    =&(a^{-1}ab)\cdot p\\
    =&b\cdot p\\
    =&\varphi_p(b),
\end{align*}
thus, $\phi_{a^{-1}}\circ\varphi_p\circ L_{a}=\varphi_p,\ \forall a\in G,$ so that $(d\phi_a)_{a\cdot p}\circ(d\varphi_p)_a\circ(dL_a)_e=(d\varphi_p)_e.$ Since $(d\phi_a)_{a\cdot p}$ and $(dL_a)_e$ are linear isomorphisms, then $$\textnormal{rank}\ (d\varphi_p)_a=\textnormal{rank}\ (d\varphi_p)_e$$
for all $a\in G.$ This means that $\varphi_p$ is a smooth map with constant rank and it is surjective because of the transitivity of the action $\phi$. Hence, $\varphi_p$ is a surjective submersion.
\end{proof}
Fix an element $o\in M$ and denote by $H=\{a\in G:a\cdot o=o\}$ the corresponding isotropy subgroup.  A homogeneous $G$-space $M$ can be identified with the quotient $G/H$ since the map 
\begin{equation*}
    \begin{array}{rccc}
    \Psi: & M & \longrightarrow & G/H\\
          & a\cdot o & \longmapsto & aH
    \end{array}
\end{equation*}
is a diffeomorphism. The isotropy representation of $M$ with respect to $o$ is the group homomorphism $j:H\longrightarrow \textnormal{GL}(T_oM),\ j(h):=(d\phi_h)_o,$ where $T_oM$ is the tangent space of $M$ at $o$, $\textnormal{GL}(T_oM)$ is the general linear group of $T_oM$ and $(d\phi_h)_o$ is the derivative of the map $\phi_h$ at $o.$ 

\begin{remark} We know that $(d\phi_h)_o$ is an automorphism of the linear space $T_oM$ because $\phi_a$ is a diffeomorphism for each $a\in G$ and, in particular, for each $h\in H.$ 
\end{remark}

Denote the Lie algebras of $G$ and $H$ by $\mathfrak{g}$ and $\mathfrak{h}$ respectively. We say that the homogeneous $G$-space $M$ is reductive if there exists a linear subspace $\mathfrak{m}$ of $\mathfrak{g}$ such that $\mathfrak{g}=\mathfrak{h}\oplus\mathfrak{m}$ and $\textnormal{Ad}(h)\mathfrak{m}=\mathfrak{m}$ for all $h\in H,$ where $\textnormal{Ad}$ is the adjoint representation of the Lie group $G.$ We also will say that $\mathfrak{g}=\mathfrak{h}\oplus\mathfrak{m}$ is a reductive decomposition of the Lie algebra $\mathfrak{g}.$ 

\begin{remark} If $\mathfrak{g}=\mathfrak{h}\oplus \mathfrak{m}$ is a reductive decomposition of the Lie algebra $\mathfrak{g}$ of $G,$ then the condition $\textnormal{Ad}(h)\mathfrak{m}=\mathfrak{m}$ implies that the map
$$\textnormal{Ad}^H\big{|}_{\mathfrak{m}}:H\longrightarrow \textnormal{GL}(\mathfrak{m});\ h\mapsto \textnormal{Ad}(h)\big{|}_\mathfrak{m}:\mathfrak{m}\rightarrow \mathfrak{m}$$
is a well-defined representation of $H$ in $\mathfrak{m}.$
\end{remark}
\begin{proposition}\label{Ad:j:equivalence}
    Let $M$ be a reductive homogeneous $G$-space, $o\in M$ and $H$ the isotropy subgroup of $o$. If $\mathfrak{g}=\mathfrak{h}\oplus \mathfrak{m}$ is a reductive decomposition of the Lie algebra $\mathfrak{g}$ of $G,$ then the representation $\textnormal{Ad}^H\big{|}_{\mathfrak{m}}$ is equivalent to the isotropy representation $j$ of $M$ with respect to $o.$
\end{proposition}
\begin{proof} We shall prove that the restriction map $(d\varphi_o)_e\big{|}_{\mathfrak{m}}:\mathfrak{m}\rightarrow T_o M$, where $(d\varphi_o)_e$ is the derivative of $\varphi_o$ at the identity element $e\in G,$ is an intertwining isomorphism between $\textnormal{Ad}^H\big{|}_{\mathfrak{m}}$ and $j.$ First, observe that for every $X\in\mathfrak{g}$ we have that
\begin{align*}
    (d\varphi_o)_e(X)=& (d\varphi_o)_e\left(\left.\frac{d}{dt}\exp(tX)\right|_{t=0}\right)\\
    \\
    =& \left.\frac{d}{dt}\varphi_o\left(\exp(tX)\right)\right|_{t=0}\\
    \\
    =&\left.\frac{d}{dt}\left[\exp(tX)\cdot o\right]\right|_{t=0},
    \end{align*}
in particular, if $X\in\mathfrak{h}$, then $\exp(tX)\in H$ and $$(d\varphi_o)_e(X)=\left.\frac{d}{dt}[\exp(tX)\cdot o]\right|_{t=0}=\left.\frac{d}{dt}o\right|_{t=0}=0,$$
thus $\mathfrak{h}\subseteq\textnormal{ker}(d\varphi_o)_e$ and, since $\mathfrak{g}=\mathfrak{h}\oplus\mathfrak{m}$, then $\textnormal{Im} (d\varphi_o)_e=\textnormal{Im}(d\varphi_o)_e\big{|}_{\mathfrak{m}}.$ On the other hand, we already know that $\varphi_o$ is a submersion (see Proposition \ref{diff:submersion}), so $$\textnormal{Im}(d\varphi_o)_e\big{|}_{\mathfrak{m}}=\textnormal{Im}(d\varphi_o)_e= T_oM.$$ 
This proves that $(d\varphi_o)_e\big{|}_{\mathfrak{m}}$ is onto, but
$$\dim \mathfrak{m}=\dim\mathfrak{g}-\dim\mathfrak{h}=\dim (G/H)=\dim M=\dim T_o M,$$ so  $(d\varphi_o)_e\big{|}_{\mathfrak{m}}:\mathfrak{m}\rightarrow T_o M$ is a linear isomorphism. Additionally, observe that for each $h\in H$ and for all $b\in G,$
\begin{align*}
    (\varphi_o\circ C_h)(b)=&\varphi_o(hbh^{-1})\hspace{1cm} (\textnormal{where } C_h \textnormal{ is the conjugation map with respect to } h)\\
    =&(hbh^{-1})\cdot o\\
    =& (hb)\cdot(h^{-1}\cdot o)\\
    =&(hb)\cdot o\hspace{1.52cm}(\textnormal{since } h\in H)\\
    =&h\cdot(b\cdot o)\\
    =&h\cdot \varphi_o(b)\\
    =&\phi_h(\varphi_o(b))\\
    =&(\phi_h\circ \varphi_o)(b),
\end{align*}
hence $$\varphi_o\circ C_h=\phi_h\circ \varphi_o,\ \forall h\in H.$$
By taking derivatives we obtain $$(d\varphi_o)_e\circ(dC_h)_e=(d\phi_h)_o\circ(d\varphi_o)_e,\ \forall h\in H$$ or, equivalently, \begin{equation}\label{equivariant:equality}(d\varphi_o)_e\circ\textnormal{Ad}(h)=j(h)\circ(d\varphi_o)_e,\ \forall h\in H.\end{equation}
In particular, $$(d\varphi_o)_e\big{|}_{\mathfrak{m}}\circ\textnormal{Ad}(h)\big{|}_{\mathfrak{m}}=j(h)\circ(d\varphi_o)_e\big{|}_{\mathfrak{m}},\ \forall h\in H,$$ thus, $(d\varphi_o)_e\big{|}_{\mathfrak{m}}$ is an intertwining operator beetween $j$ and $\Ad^H\big{|}_{\mathfrak{m}}$. The proof is complete. 
\end{proof}

We say that a Riemannian metric $g$ on a homogeneous reductive $G$-space $M$ is $G$-invariant whenever $\phi_a:(M,g)\rightarrow (M,g)$ is an isometry for every $a\in G.$ Such a metric can be identified with an $\textnormal{Ad}(H)$-invariant inner product defined on $\mathfrak{m},$ that is, an inner product for which $\textnormal{Ad}(h)\in \textnormal{O}(\mathfrak{m}),\ \forall h\in H.$  More precisely, we have the following proposition:
\begin{proposition}\label{metric:inner:equivalence} Let $M=G/H$ be a reductive homogeneous $G$-space and $\mathfrak{g}=\mathfrak{h}\oplus\mathfrak{m}$ a reductive decomposition of the Lie algebra $\mathfrak{g}$ of $G.$ Define the sets
$$\textnormal{Inv}_G(M):=\{g: g\textnormal{ is a $G$-invariant Riemannian metric on $M$}\}$$ and $$\textnormal{Inn}_H(\mathfrak{m}):=\{\langle\cdot,\cdot\rangle:\langle\cdot,\cdot\rangle \textnormal{ is an Ad($H$)-invariant inner product on $\mathfrak{m}$} \}.$$ Then, the map $g\in\textnormal{Inv}_G(M)\mapsto \langle\cdot,\cdot\rangle_g\in\textnormal{Inn}_H(\mathfrak{m}),$  where $$\langle X,Y\rangle_g:=g_o((d\varphi_o)_e(X),(d\varphi_o)_e(Y)),$$ is a bijection.
\end{proposition}
\begin{proof} Let $g\in \textnormal{Inv}_g(M).$ Since $(d\varphi_o)_e\big{|}_{\mathfrak{m}}$ is a linear isomorphism, then $\langle\cdot,\cdot\rangle_g$ is an inner product on $\mathfrak{m}.$ Additionally, for every $h\in H$ and $X,Y\in\mathfrak{m}$ we have that
\begin{align*}
    \langle \textnormal{Ad}(h)X,\Ad(h)Y\rangle_g=&g_o((d\varphi_o)_e\circ\Ad(h)(X),(d\varphi_o)_e\circ\Ad(h)(Y))\\
    =&g_o(j(h)\circ(d\varphi_o)_e(X),j(h)\circ(d\varphi_o)_e(Y)),\hspace{2cm} (\textnormal{by} \ \eqref{equivariant:equality})\\
    =&g_o((d\phi_h)_o\circ(d\varphi_o)_e(X),(d\phi_h)_o\circ(d\varphi_o)_e(Y))\\
    =&(\phi_h^*g)_o((d\varphi_o)_e(X),(d\varphi_o)_e(Y))\\
    =&g_o((d\varphi_o)_e(X),(d\varphi_o)_e(Y))\\
    =&\langle X,Y\rangle_g,
\end{align*}
where we use the fact that $\phi_h^*g=g$ because $g$ is $G$-invariant. So far, we have shown that the map $g\in\textnormal{Inv}_G(M)\mapsto \langle\cdot,\cdot\rangle_g\in\textnormal{Inn}_H(\mathfrak{m})$ is well-defined. Now, if $\langle\cdot,\cdot\rangle_{g_1}=\langle\cdot,\cdot\rangle_{g_2}$ for $g_1,g_2\in\textnormal{Inv}_G(M),$ then for any $p=a\cdot o\in M$ and $\tilde{X},\tilde{Y}\in T_pM$ we have
\begin{align*}
    (g_1)_p(\tilde{X},\tilde{Y})=&(\phi_{a^{-1}}^\ast g_1)_p(\tilde{X},\tilde{Y})\\
                =&(g_1)_{o}((d\phi_{a^{-1}})_p(\tilde{X}),(d\phi_{a^{-1}})_p(\tilde{Y}))\\
                =&(g_1)_o(X,Y),\ (\textnormal{where }X:=(d\phi_{a^{-1}})_p(\tilde{X})\ \textnormal{and }Y:=(d\phi_{a^{-1}})_p(\tilde{Y}))\\
                =&(g_1)_o\left((d\varphi_o)_e\left[\left((d\varphi_o)_e\big{|}_{\mathfrak{m}}\right)^{-1}(X)\right],(d\varphi_o)_e\left[\left((d\varphi_o)_e\big{|}_{\mathfrak{m}}\right)^{-1}(Y)\right]\right)\\
                =&\left\langle \left((d\varphi_o)_e\big{|}_{\mathfrak{m}}\right)^{-1}(X),\left((d\varphi_o)_e\big{|}_{\mathfrak{m}}\right)^{-1}(Y)\right\rangle_{g_1}\\
                =&\left\langle \left((d\varphi_o)_e\big{|}_{\mathfrak{m}}\right)^{-1}(X),\left((d\varphi_o)_e\big{|}_{\mathfrak{m}}\right)^{-1}(Y)\right\rangle_{g_2}\\
                =&(g_2)_o\left((d\varphi_o)_e\left[\left((d\varphi_o)_e\big{|}_{\mathfrak{m}}\right)^{-1}(X)\right],(d\varphi_o)_e\left[\left((d\varphi_o)_e\big{|}_{\mathfrak{m}}\right)^{-1}(Y)\right]\right)\\
                =&(g_2)_o(X,Y)\\
                =&(g_2)_{o}((d\phi_{a^{-1}})_p(\tilde{X}),(d\phi_{a^{-1}})_p(\tilde{Y}))\\
                =&(\phi_{a^{-1}}^\ast g_2)_p(\tilde{X},\tilde{Y})\\
                =&(g_2)_p(\tilde{X},\tilde{Y}),
\end{align*}
so $g_1=g_2.$ This shows that the map $g\in\textnormal{Inv}_G(M)\mapsto \langle\cdot,\cdot\rangle_g\in\textnormal{Inn}_H(\mathfrak{m})$ is injective. Finally, given an $\Ad(H)$-invariant inner product $\langle\cdot,\cdot\rangle$ on $\mathfrak{m},$ define a Riemannian metric $g$ on $M$ by $$g_p(\tilde{X},\tilde{Y}):=\left\langle \left((d\varphi_o)_e\big{|}_{\mathfrak{m}}\right)^{-1}\left((d\phi_{a^{-1}})_p(\tilde{X})\right),\left((d\varphi_o)_e\big{|}_{\mathfrak{m}}\right)^{-1}\left((d\phi_{a^{-1}})_p(\tilde{Y})\right)\right\rangle,$$ 
for $p=a\cdot o\in M$ and $\tilde{X},\tilde{Y}\in T_pM.$ By its own definition, we have that $g$ is $G$-invariant; also, for $X,Y\in\mathfrak{m},$ 
\begin{align*}
    \left\langle X,Y\right\rangle_g=&g_o((d\varphi_o)_e(X),(d\varphi_o)_e(Y))\\
    =&\left\langle \left((d\varphi_o)_e\big{|}_{\mathfrak{m}}\right)^{-1}(d\phi_{e^{-1}})_p\left((d\varphi_o)_e(X)\right),\left((d\varphi_o)_e\big{|}_{\mathfrak{m}}\right)^{-1}(d\phi_{e^{-1}})_p\left((d\varphi_o)_e(Y)\right)\right\rangle\\
    =&\left\langle \left((d\varphi_o)_e\big{|}_{\mathfrak{m}}\right)^{-1}(d\varphi_o)_e(X),\left((d\varphi_o)_e\big{|}_{\mathfrak{m}}\right)^{-1}(d\varphi_o)_e(Y)\right\rangle\\
    =&\langle X, Y\rangle.
\end{align*}
Hence, $g\in\textnormal{Inv}_G(M)\mapsto \langle\cdot,\cdot\rangle_g\in\textnormal{Inn}_H(\mathfrak{m})$ is surjective. The proof is complete.
\end{proof}

From now on, we will assume that $G$ is compact so that there exists a unique (up to re-scaling) $\Ad(G)$-invariant inner product $(\cdot, \cdot)$ on $\mathfrak{g}.$ This inner product induces a reductive decomposition $\mathfrak{g}=\mathfrak{h}\oplus \mathfrak{m},$ where $\mathfrak{m}$ is the orthogonal complement of $\mathfrak{h}$ in $\mathfrak{g}.$ The restriction $(\cdot,\cdot)\big{|}_{\mathfrak{m}\times\mathfrak{m}}$ of $(\cdot,\cdot)$ to $\mathfrak{m}$ defines a $\Ad(H)$-invariant inner product on $\mathfrak{m},$ that is, a $G$-invariant metric on $M=G/H$ (see Proposition \ref{metric:inner:equivalence}). If $g$ is an arbitrary $G$-invariant metric on $M$ and $\langle\cdot,\cdot\rangle_g$ is the corresponding $\Ad(H)$-invariant inner product on $\mathfrak{m}$ induced by $g,$ then there exists a unique linear operator $A:\mathfrak{m}\rightarrow\mathfrak{m}$ such that \begin{equation}\label{metric:operator:def}(AX,Y)=\langle X,Y\rangle_g,\ \forall X,Y\in\mathfrak{m}.\end{equation} We will say that $A$ is the metric operator associated with $g$ and, when $(\cdot,\cdot)$ is fixed, we can identify any invariant metric with its associated metric operator. The following proposition provides some properties of a metric operator.

\begin{proposition}\label{properties:metric:operator} Let $g$ be an invariant metric and $A:\mathfrak{m}\rightarrow\mathfrak{m}$ its associated metric operator. Then:
\begin{itemize}
    \item[$i)$] $A$ is positive-definite.\vspace{0.3cm}
    \item[$ii)$] $A$ is self-adjoint with respect to $(\cdot,\cdot)\big{|}_{\mathfrak{m}\times\mathfrak{m}}.$\vspace{0.3cm}
    \item[$iii)$] $A$ commutes with $\Ad(h)\big{|}_\mathfrak{m}$ for all $h\in H.$
\end{itemize}
Furthermore, a linear operator satisfying $i)-iii)$ is the metric operator associated with a $G$-invariant metric.
\end{proposition}
\begin{proof} Since $\langle\cdot,\cdot\rangle_g$ is an inner product then $$\langle X,X\rangle_g>0,\ X\in\mathfrak{m}-\{0\}$$ and $$\langle X,Y\rangle_g=\langle Y,X\rangle_g,\ X,Y\in\mathfrak{m}.$$ By equality \eqref{metric:operator:def} we have that
$$(AX,X)>0,\ X\in\mathfrak{m}-\{0\}$$ and $$(AX,Y)=(AY,X),\ X,Y\in\mathfrak{m}.$$
This proves $i)$ and $ii).$ On the other hand, for each $h\in H$ we have
\begin{align*}
    (AX,Y)=\ &\langle X,Y\rangle_g\\
    =\ &\langle\Ad(h)X,\Ad(h)Y\rangle_g\hspace{3.75cm} (\textnormal{since }\langle\cdot,\cdot\rangle_g\ \textnormal{is }\Ad(H)\textnormal{-invariant})\\
    =\ & (A\Ad(h)X,\Ad(h)Y)\\
    =\ & (\Ad(h^{-1}) A\Ad(h)X,\Ad(h^{-1})\Ad(h)Y)\hspace{0.5cm} (\textnormal{since }(\cdot,\cdot)\ \textnormal{is }\Ad(G)\textnormal{-invariant})\\
    =\ & (\Ad(h^{-1}) A\Ad(h)X,Y)\\
    =\ & (\Ad(h)^{-1} A\Ad(h)X,Y).
\end{align*}
This equality holds for every $X,Y\in\mathfrak{m},$ therefore $AX=\Ad(h)^{-1} A\Ad(h)X,\ \forall X\in\mathfrak{m}$ or, equivalently, $$\Ad(h) A=A\Ad(h).$$ This proves $iii).$ Now, if $A:\mathfrak{m}\rightarrow \mathfrak{m}$ satisfies $i)-iii),$ then is easy to see that the formula $$\langle X,Y\rangle:=(AX,Y)$$ defines a $\Ad(H)$-invariant inner product on $\mathfrak{m}$ such that, by virtue of Proposition \ref{metric:inner:equivalence}, induces a unique $G$-invariant metric $g$ on $M.$ By definition, $A$ is the metric operator associated with $g.$
\end{proof}
It is well known that the isotropy representation (henceforth, referred to as $\Ad^H\big{|}_\mathfrak{m}$ due to Proposition \ref{Ad:j:equivalence}) of a reductive homogeneous $G$-space is completely reducible, this means that there exist irreducible $\Ad(H)$-invariant subspaces  $\mathfrak{m}_1,...,\mathfrak{m}_s$ such that \begin{equation}\label{Isotropy:decomposition}
    \mathfrak{m}=\mathfrak{m}_1\oplus\cdots\oplus\mathfrak{m}_s.
\end{equation}
The subspaces $\mathfrak{m}_1,...,\mathfrak{m}_s$ are usually called isotropy summands of the decomposition \eqref{Isotropy:decomposition}. Since the inner product $(\cdot,\cdot)$ is $\Ad(G)$-invariant, we may assume that the subspaces $\mathfrak{m}_1,...,\mathfrak{m}_s$ are pairwise $(\cdot,\cdot)$-orthogonal. Let us consider the equivalence relation $\cong$ in the set $\{\mathfrak{m}_1,...,\mathfrak{m}_s\}$ given by
$$\mathfrak{m}_i\cong\mathfrak{m}_j\ \textnormal{if and only if}\ \Ad^H\big{|}_{\mathfrak{m}_i}\ \textnormal{and}\ \Ad^H\big{|}_{\mathfrak{m}_j}\ \textnormal{are equivalent as representations}.$$ This relation induces a partition $$\displaystyle\{1,...,s\}=C_1\cup...\cup C_S\ \textnormal{with}\ S\leq s,$$
so that
$$\mathfrak{m}=M_1\oplus\cdots\oplus M_S,$$ where $$\displaystyle M_k:=\bigoplus\limits_{j\in C_k}\mathfrak{m}_j,\ k=1,...,S.$$
The subspaces $M_1,...,M_S$ are called isotypical summands of the decomposition \eqref{Isotropy:decomposition}.
\begin{remark} The decomposition \eqref{Isotropy:decomposition} is not necessarily unique; however, it becomes unique (up to reordering) when all the subrepresentations $\mathfrak{m}_1,...,\mathfrak{m}_s$ are pairwise inequivalent.
\end{remark}
Given $j\in\{1,...,s\}$, denote by $i_j:\mathfrak{m}_j\rightarrow\mathfrak{m}$ the inclusion map and $P_j:\mathfrak{m}\rightarrow\mathfrak{m}_j$ the projection map onto $\mathfrak{m}_j$ with respect to \eqref{Isotropy:decomposition}, this means that for a given $Y=Y_1+...+Y_s\in\mathfrak{m},\   Y_r\in\mathfrak{m}_r,\ r=1,...,s;$ we have that $P_j(Y):=Y_j.$ If we define $Q_j(Y):=\sum\limits_{r\neq j}Y_r,$ then $Y=P_j(Y)+Q_j(Y),$ where $P_j^2(Y)=P_j(Y)$ and $P_jQ_j(Y)=0.$
\begin{proposition}\label{A:Equivariant} Let $A:\mathfrak{m}\rightarrow\mathfrak{m}$ be a metric operator. Then $P_k\circ A\circ i_j:\mathfrak{m}_j\rightarrow\mathfrak{m}_k$ is an equivariant map between the irreducible representations $\Ad^H\big{|}_{\mathfrak{m}_j}$ and $\Ad^H\big{|}_{\mathfrak{m}_k}$ for all $j$ and $k$ in the set $\{1,...,s\}.$  
\end{proposition}
\begin{proof} Let $j,k\in\{1,...,s\},$ $h\in H$ and  $X\in\mathfrak{m}_j.$ Then \begin{equation}\label{equivariant:1}\left[\Ad(h)\circ P_k\circ A\circ i_j\right](X)=\ \Ad(h)(P_k(AX)).\end{equation} Since $P_k(AX)\in\mathfrak{m}_k$ then $\Ad(h)(P_k(AX))\in\mathfrak{m}_k,$ thus \begin{equation}\label{P:projection}P_k(\Ad(h)(P_k(AX))=\Ad(h)(P_k(AX)).
\end{equation}
Additionally, $\displaystyle Q_k(AX)\in\bigoplus\limits_{r\neq k}\mathfrak{m}_r$ (which is also an $\Ad(H)$-invariant subspace as each $\mathfrak{m}_r$ is), hence $$\displaystyle\Ad(h)(Q_k(AX))\in\bigoplus\limits_{r\neq k}\mathfrak{m}_r$$ which implies that \begin{equation}\label{Q:projection}P_k(\Ad(h)(Q_k(AX)))=0.
\end{equation}
By adding equations \eqref{P:projection} and \eqref{Q:projection} we obtain
\begin{align*}
    &P_k(\Ad(h)(P_k(AX))+P_k(\Ad(h)(Q_k(AX)))=\Ad(h)(P_k(AX))\\
    \Longrightarrow\ &P_k(\Ad(h)(P_k(AX)+Q_k(AX)))=\Ad(h)(P_k(AX))\\
    \Longrightarrow\ &P_k(\Ad(h)(AX))=\Ad(h)(P_k(AX)),
\end{align*}
but $A$ is a metric operator, so $\Ad(h)(AX)=A(\Ad(h)X)$ (see Proposition \ref{properties:metric:operator}, $iii)$) and, therefore
\begin{equation}\label{auxiliar:P:A:Ad}
    P_k(A(\Ad(h)X))=\Ad(h)(P_k(AX)).
\end{equation}
Finally, note that $X\in\mathfrak{m}_j,$ which is $\Ad(H)$-invariant, then $\Ad(h)X\in\mathfrak{m}_j$ and $i_j(\Ad(h)X)=\Ad(h)X.$ Hence, we can conclude from \eqref{auxiliar:P:A:Ad} that
$$P_k(A(i_j(\Ad(h)X)))=\Ad(h)(P_k(AX)),$$ or equivalently, \begin{equation}\label{equivariant:2}
    \left[P_k\circ A\circ i_j\circ\Ad(h)\right](X)=\Ad(h)(P_k(AX)).
\end{equation}
By \eqref{equivariant:1} and \eqref{equivariant:2} we obtain 
\begin{equation*}
    \left[P_k\circ A\circ i_j\circ\Ad(h)\right](X)=\left[\Ad(h)\circ P_k\circ A\circ i_j\right](X).
\end{equation*}
Since $X\in\mathfrak{m}_j$ was arbitrarily chosen, then $$P_k\circ A\circ i_j\circ\Ad(h)\big{|}_{\mathfrak{m}_j}=\Ad(h)\big{|}_{\mathfrak{m}_k}\circ P_k\circ A\circ i_j.$$
The proof is complete.
\end{proof}
\begin{corollary}\label{j:k:0} If $\Ad^H\big{|}_{\mathfrak{m}_j}$ and $\Ad^H\big{|}_{\mathfrak{m}_k}$ are not equivalent, then $P_k\circ A\circ i_j=0.$ 
\end{corollary}
\begin{proof} By Proposition \ref{A:Equivariant}, $P_k\circ A\circ i_j$ is an equivariant map between the representations of $H$ in $\mathfrak{m}_j$ and $\mathfrak{m}_k.$ Since these representations are irreducible, then either $P_k\circ A\circ i_j=0$ or $P_k\circ A\circ i_j$ is an isomorphism. If $\Ad^H\big{|}_{\mathfrak{m}_j}$ and $\Ad^H\big{|}_{\mathfrak{m}_k}$ are not equivalent representations, then $P_k\circ A\circ i_j$ cannot be an isomorphism, hence $P_k\circ A\circ i_j=0.$
\end{proof}
The next result is a useful theorem that characterizes all the equivariant endomorphisms of an irreducible real representation. We refer to \cite[Chapter II, Theorem 6.3]{BD} for a proof.
\begin{theorem}\label{irreducible:real:representations}
    Let $\rho:K\rightarrow \textnormal{GL}(V)$ be an irreducible real representation of a compact Lie group and let $(\cdot,\cdot)_V$ be a $\rho(K)$-invariant inner product on $V.$ Then there exists a $(\cdot,\cdot)_V$-orthonormal basis $\mathcal{B}_V$ of $V$ such that one of the following assertions holds:
    \begin{itemize}
        \item[$i)$] For each equivariant endomorphism $S:V\rightarrow V,$  there exist $\zeta\in\mathbb{R}$ such that 
    \begin{equation}\label{end:orthogonal}
 [S]_{\mathcal{B}_V}=\left(\begin{array}{cccc}
            \zeta & 0 & \cdots & 0\\
            0 & \zeta & \cdots & 0\\
            \vdots & \vdots & \ddots & \vdots\\
            0 & 0 & \cdots & \zeta
            \end{array}\right),
    \end{equation}
    and, conversely, for each $\zeta \in\mathbb{R},$ the formula \eqref{end:orthogonal} defines an equivariant endomorphism $S:V\rightarrow V.$\\
    
    \item[$ii)$] For each equivariant endomorphism $S:V\rightarrow V,$ there exists a matrix $$\zeta=\left(\begin{array}{cc}a & -b \\ b & a \end{array}\right),\ a,b\in\mathbb{R},$$ such that
    \begin{equation}\label{end:complex}
        [S]_{\mathcal{B}_V}=\left(\begin{array}{cccc}
            \zeta & 0 & \cdots & 0\\
            0 & \zeta & \cdots & 0\\
            \vdots & \vdots & \ddots & \vdots\\
            0 & 0 & \cdots & \zeta
            \end{array}\right),
    \end{equation}
    and, conversely, for each matrix $$\zeta=\left(\begin{array}{cc}a & -b \\ b & a \end{array}\right),\ a,b\in\mathbb{R},$$ the formula \eqref{end:complex} defines an equivariant endomorphism $S:V\rightarrow V.$\\
    
    \item[$iii)$] For each equivariant endomorphism $S:V\rightarrow V,$ there exists a matrix $$\zeta=\left(\begin{array}{cccc}a & d & -b & -c \\ -d & a & c & -b\\ b & -c & a & -d\\ c & b & d & a \end{array}\right),\ a,b,c,d\in\mathbb{R},$$ such that
    \begin{equation}\label{end:quaternionic}
        [S]_{\mathcal{B}_V}=\left(\begin{array}{cccc}
            \zeta & 0 & \cdots & 0\\
            0 & \zeta & \cdots & 0\\
            \vdots & \vdots & \ddots & \vdots\\
            0 & 0 & \cdots & \zeta
            \end{array}\right),
    \end{equation}
    and, conversely, for each matrix $$\zeta=\left(\begin{array}{cccc}a & d & -b & -c \\ -d & a & c & -b\\ b & -c & a & -d\\ c & b & d & a \end{array}\right),\ a,b,c,d\in\mathbb{R}$$ the formula \eqref{end:quaternionic} defines an equivariant endomorphism $S:V\rightarrow V.$
    \end{itemize}
    \end{theorem}
    \begin{corollary}\label{invariant:inner:products}
    Let $\rho:K\rightarrow\textnormal{GL}(V)$ be a real irreducible representation of a compact Lie group, and let $(\cdot,\cdot)_1$ and $(\cdot,\cdot)_2$  be $\rho$-invariant inner products on $V.$ Then, there exists $a>0$ such that $(\cdot,\cdot)_1=a(\cdot,\cdot)_2.$
    \end{corollary}
    \begin{proof} Since $(\cdot,\cdot)_1$ and $(\cdot,\cdot)_2$ are $\rho$-invariant inner products, there exists a linear map $S:V\rightarrow V$ such that \begin{equation}\label{auxiliar:inner:products}(u,v)_1=(Su,v)_2,\ u,v\in V.\end{equation} By proceeding analogously to the proof of Proposition \ref{properties:metric:operator}, we can show that $S$ is positive-definite, self-adjoint with respect to $(\cdot,\cdot)_2$ and commutes with $\rho(k)$ for every $k\in K,$ that is, $S$ is $\rho$-equivariant. By Theorem \ref{irreducible:real:representations}, there exists a $(\cdot,\cdot)_2$-orthonormal basis $\mathcal{B}_2$ of $V$ such that 
        $$[S]_{\mathcal{B}_2}=\left(\begin{array}{cccc}
            \zeta & 0 & \cdots & 0\\
            0 & \zeta & \cdots & 0\\
            \vdots & \vdots & \ddots & \vdots\\
            0 & 0 & \cdots & \zeta
            \end{array}\right),$$
    where either $\zeta\in\mathbb{R},$ $$\zeta=\left(\begin{array}{cc}a & -b \\ b & a \end{array}\right),\ a,b\in\mathbb{R},$$ or $$\zeta=\left(\begin{array}{cccc}a & d & -b & -c \\ -d & a & c & -b\\ b & -c & a & -d\\ c & b & d & a \end{array}\right),\ a,b,c,d\in\mathbb{R}.$$ We shall show that $\zeta$ is a scalar multiple of the identity matrix. In fact, if $\zeta\in\mathbb{R}$, then $\left[S\right]_{\mathcal{B}_2}-a\textnormal{I}$ is skew-symmetric, but it is also symmetric because $S$ is self-adjoint with respect to $(\cdot,\cdot)_2$ and the basis $\mathcal{B}_2$ is $(\cdot,\cdot)_2$-orthonormal. Thus $\left[S\right]_{\mathcal{B}_2}-a\textnormal{I}$ must be zero and $S=a\textnormal{I}.$ On the other hand, $S$ is positive-definite so $a>0.$ By formula \eqref{auxiliar:inner:products}, it follows that $$(u,v)_1=(Su,v)_2=(a u,v)_2=a(u,v)_2.$$ The proof is complete.
    \end{proof}
    \begin{corollary}\label{intertwining:isometry} 
    Let $\rho:K\rightarrow\textnormal{GL}(V)$ and $\tau:K\rightarrow \textnormal{GL}(W)$ be equivalent real irreducible representations of a compact Lie group $K,$ let $(\cdot,\cdot)_V$ be a $\rho$-invariant inner product on $V$ and let $(\cdot,\cdot)_W$ be a $\tau$-invariant inner product on $W.$ Then, there exists a intertwining operator $T:V\rightarrow W$ preserving these inner products, that is, $$(Tu,Tv)_W=(u,v)_V,\ u,v\in V.$$
    \end{corollary}
    \begin{proof} Since $\rho$ and $\tau$ are equivalent representations, there exists a intertwining isomorphism $\tilde{T}:V\rightarrow W.$ Let $\mathcal{B}_V=\{v_1,...,v_d\}$ be a $(\cdot,\cdot)_V$-orthonormal basis of $V$ as in Theorem \ref{irreducible:real:representations}. Define an inner product $(\cdot,\cdot)_0$ by declaring the set $\mathcal{B}_0=\{\tilde{T}v_1,...,\tilde{T}v_d\}$ to be a $(\cdot,\cdot)_0$-orthonormal  basis. Clearly,
    \begin{equation}\label{auxiliar:tilde:isometry}\left(\tilde{T}u,\tilde{T}v\right)_0=(u,v)_V,\ u,v\in V.\end{equation} We shall show that $(\cdot,\cdot)_0$ is $\tau$-invariant. In fact, given $k\in K$ and $w_1,w_2\in W,$ we have that
    \begin{align*}
        (\tau(k)w_1,\tau(k)w_2)_0 =&\left(\tau(k)\tilde{T}\tilde{T}^{-1}w_1,\tau(k)\tilde{T}\tilde{T}^{-1}w_2\right)_0\\
        =&\left(\tilde{T}\rho(k)\tilde{T}^{-1}w_1,\tilde{T}\rho(k)\tilde{T}^{-1}w_2\right)_0\\
        =&\left(\rho(k)\tilde{T}^{-1}w_1,\rho(k)\tilde{T}^{-1}w_2\right)_V\hspace{1.44cm} \textnormal{by}\ \eqref{auxiliar:tilde:isometry}\\
        =&\left(\tilde{T}^{-1}w_1,\tilde{T}^{-1}w_2\right)_V\hspace{3cm} \textnormal{since}\ (\cdot,\cdot)_V\ \textnormal{is}\ \rho\textnormal{-invariant}\\
        =&\left(\tilde{T}\tilde{T}^{-1}w_1,\tilde{T}\tilde{T}^{-1}w_2\right)_0\hspace{2.49cm} \textnormal{by}\ \eqref{auxiliar:tilde:isometry}\\
        =&\left(w_1,w_2\right)_0.
    \end{align*}
        This means that $(\cdot,\cdot)_0$ is $\tau$-invariant so $(\cdot,\cdot)_0=a(\cdot,\cdot)_W$ for some $a>0$ (see Corollary \ref{invariant:inner:products}). Since $\tilde{T}$ is an intertwining operator, then $T:=\sqrt{a}\tilde{T}$ is also a intertwining operator, furthermore,
        $$(Tu,Tv)_W=\left(\sqrt{a}\tilde{T}u,\sqrt{a}\tilde{T}v\right)_W=a\left(\tilde{T}u,\tilde{T}v\right)_W=\left(\tilde{T}u,\tilde{T}v\right)_0=(u,v)_V.$$
    \end{proof}
    
    \begin{corollary}\label{equivariant:maps} Let $\rho:K\rightarrow\textnormal{GL}(V)$ and $\tau:K\rightarrow \textnormal{GL}(W)$ be equivalent real irreducible representations of a compact Lie group $K,$ let $(\cdot,\cdot)_V$ be a $\rho$-invariant inner product on $V$ and let $(\cdot,\cdot)_W$ be a $\tau$-invariant inner product on $W.$ Let $T:V\rightarrow W$ be an intertwining operator preserving these inner products. Let $\mathcal{B}_V$ a basis of $V$ as in Theorem \ref{irreducible:real:representations} and $\mathcal{B}_W:=T(\mathcal{B}_V).$ Then one of the following assertions holds: 
    \begin{itemize}
        \item[$i)$] Any equivariant map $S:V\rightarrow W$ can be written in the basis $\mathcal{B}_V$ and $\mathcal{B}_W$ as
    \begin{equation*}
        [S]_{\mathcal{B}_V\mathcal{B}_W}=\left(\begin{array}{cccc}
            \zeta & 0 & \cdots & 0\\
            0 & \zeta & \cdots & 0\\
            \vdots & \vdots & \ddots & \vdots\\
            0 & 0 & \cdots & \zeta
            \end{array}\right),\ \zeta\in\mathbb{R}.
    \end{equation*}
    \item[$ii)$] Any equivariant map $S:V\rightarrow W$ can be written in the bases $\mathcal{B}_V$ and $\mathcal{B}_W$ as
    \begin{equation*}
        [S]_{\mathcal{B}_V\mathcal{B}_W}=\left(\begin{array}{cccc}
            \zeta & 0 & \cdots & 0\\
            0 & \zeta & \cdots & 0\\
            \vdots & \vdots & \ddots & \vdots\\
            0 & 0 & \cdots & \zeta
            \end{array}\right),\ \zeta=\left(\begin{array}{cc}a & -b \\ b & a \end{array}\right),\ a,b\in\mathbb{R}.
    \end{equation*}
    \item[$iii)$] Any equivariant map $S:V\rightarrow W$ can be written in the bases $\mathcal{B}_V$ and $\mathcal{B}_W$ as
    \begin{equation*}
        \hspace{1.3cm}[S]_{\mathcal{B}_V\mathcal{B}_W}=\left(\begin{array}{cccc}
            \zeta & 0 & \cdots & 0\\
            0 & \zeta & \cdots & 0\\
            \vdots & \vdots & \ddots & \vdots\\
            0 & 0 & \cdots & \zeta
            \end{array}\right),\ \zeta=\left(\begin{array}{cccc}a & d & -b & -c \\ -d & a & c & -b\\ b & -c & a & -d\\ c & b & d & a \end{array}\right),\ a,b,c,d\in\mathbb{R}.
    \end{equation*}
    \end{itemize}
    \end{corollary}
    \begin{proof} For any equivariant map $S:V\rightarrow W$ we have that $T^{-1}S:V\rightarrow V$ is an equivariant endomorphism of $V.$ Additionally, since $\mathcal{B}_W=T(\mathcal{B}_V)$ then $\left[T^{-1}\right]_{\mathcal{B}_W\mathcal{B}_V}=\textnormal{I}_m$, where $\textnormal{I}_m$ is the identity matrix of order $m=\dim W.$ Therefore,
        $$\left[T^{-1}S\right]_{\mathcal{B}_V}=\left[T^{-1}\right]_{\mathcal{B}_W\mathcal{B}_V}\left[S\right]_{\mathcal{B}_V\mathcal{B}_W}=\textnormal{I}_m\left[S\right]_{\mathcal{B}_V\mathcal{B}_W}=\left[S\right]_{\mathcal{B}_V\mathcal{B}_W}.$$
    From this last equality and Theorem \ref{irreducible:real:representations} follows the result. 
    \end{proof}
    \begin{remark}\label{matrix:representation:zeta} The complex number $a+ib$ and the quaternion $a+bi+cj+dk$ can be represented by the square matrices $$\left(\begin{array}{cc} a & -b \\ b & a \end{array}\right)\ \textnormal{and}\ \left(\begin{array}{cccc}a & d & -b & -c \\ -d & a & c & -b\\ b & -c & a & -d\\ c & b & d & a \end{array}\right)$$ respectively. Corollary \ref{equivariant:maps} essentially says that for irreducible real representations $\rho:K\rightarrow \textnormal{GL}(V)$ and $\tau:K\rightarrow \textnormal{GL}(W)$ of a compact Lie group, the set of all equivariant maps from $V$ to $W$ is isomorphic $\mathbb{R},$ $\mathbb{C}$ or $\mathbb{H},$ in symbols, $$\textnormal{Hom}_{\rho,\tau}(V,W)\cong\mathbb{R},\mathbb{C},\ \textnormal{or}\ \mathbb{H}.$$  
    \end{remark}
    We may apply Corollary \ref{equivariant:maps} to describe the metric operator $A$ associated with a $G$-invariant metric on a homogeneous $G$-space $M.$ Recall that we fixed a $\Ad(G)$-invariant inner product $(\cdot,\cdot)$ on $\mathfrak{g}.$ This inner product, when restricted to an isotropy summand $\mathfrak{m}_j,$ gives rise to an $\Ad(H)$-invariant inner product on $\mathfrak{m}_j.$ Unless the context lacks clarity, we will use the notation $(\cdot,\cdot)$ to refer to the restriction of $(\cdot,\cdot)$ to $\mathfrak{m}_j\times \mathfrak{m}_j$.
    \begin{lemma}\label{almost:main} Let $$\displaystyle M_i=\bigoplus\limits_{j\in C_i}\mathfrak{m}_j=\mathfrak{m}_{j_1}\oplus\cdots\oplus\mathfrak{m}_{j_t}$$ be an isotypical summand of the decomposition \eqref{Isotropy:decomposition}. Fix a $(\cdot,\cdot)$-orthonormal basis $\mathcal{B}_{j_1}$ of $\mathfrak{m}_{j_1}$ as in Theorem \ref{irreducible:real:representations}, and also fix linear isomorphisms $T_{j_1}^{j_k}:\mathfrak{m}_{j_1}\rightarrow\mathfrak{m}_{j_k},\ k=1,2,...,t$ such that:
    \begin{itemize}
        \item[$i)$] $T_{j_1}^{j_1}$ is the identity map of $\mathfrak{m}_{j_1}.$
        \item[$ii)$] $T_{j_1}^{j_k}$ is an intertwining operator between $\Ad^H\big{|}_
        {\mathfrak{m}_{j_1}}$ and $\Ad^H\big{|}_
        {\mathfrak{m}_{j_k}},$ for $k=2,...,t.$
        \item[$iii)$] $T_{j_1}^{j_k}$ preserves $(\cdot,\cdot)$ for $k=2,....,t.$
    \end{itemize} (note that these isomorphisms exist due to Corollary \ref{intertwining:isometry}). Let $\mathcal{B}_{j_k}:=T_{j_1}^{j_k}(\mathcal{B}_{j_1}),\ k=1,...,t.$ Then, for any metric operator $A$, we have that
    \begin{equation}\label{blocks:of:A}
        \left[P_{j_k}\circ A\circ i_{j_l}\right]_{\mathcal{B}_{j_l}\mathcal{B}_{j_k}}=\left(\begin{array}{cccc}
            \zeta_{j_l}^{j_k} & 0 & \cdots & \\
            0 & \zeta_{j_l}^{j_k} & \cdots & 0\\
            \vdots & \vdots & \ddots & \vdots\\
            0 & 0 & \cdots & \zeta_{j_l}^{j_k}
            \end{array}\right),\ l,k\in\{1,...,t\},
    \end{equation}
    where each $\zeta_{j_l}^{j_k}$ is either a real number, the matrix representation of a complex number, or the matrix representation of a quaternion (see Remark \ref{matrix:representation:zeta}). 
    \end{lemma}
    \begin{proof} Let $l,k\in\{1,...,t\}.$ By Proposition \ref{A:Equivariant} we have that $P_{j_k}\circ A\circ i_{j_l}$ is an equivariant map between the representations $\Ad^H\big{|}_{\mathfrak{m}_{j_l}}$ and $\Ad^H\big{|}_{\mathfrak{m}_{j_k}}.$ In addition, $$T_{j_l}^{j_k}:=T_{j_1}^{j_k}\circ\left(T_{j_1}^{j_l}\right)^{-1}:\mathfrak{m}_{j_l}\rightarrow\mathfrak{m}_{j_k}$$ is a intertwining isomorphism and $$T_{j_l}^{j_k}(\mathcal{B}_{j_l})=\mathcal{B}_{j_k}.$$ Hence, equation \eqref{blocks:of:A} is valid due to Corollary \ref{equivariant:maps}. 
    \end{proof}

    \begin{remark} With the notations introduced in the preceding Lemma, it is important to highlight that, for every $k\in\{1,...,t\},$ the set $\mathcal{B}_{j_k}$ is indeed a $(\cdot,\cdot)$-orthonormal basis of $\mathfrak{m}_{j_k}.$ This is a result of two observations: the orthonormality of $\mathcal{B}_{j_1}$ with respect to the inner product $(\cdot,\cdot),$ and the property that the map $T_{j_1}^{j_k}$ preserves $(\cdot,\cdot).$ Furthermore, the union $\tilde{B}_i:=\mathcal{B}_{j_1}\cup\cdots\cup\mathcal{B}_{j_t}$ forms a basis for the isotypical summand $M_i$. This procedure can be extended to all values of $i$ within the set $\{1,...,S\},$ leading to corresponding sets $\tilde{\mathcal{B}}_1,...,\tilde{\mathcal{B}}_S,$ each forming a basis for the respective isotypical summands $M_1,...,M_S.$ Since the isotropy summands $\mathfrak{m}_1,...,\mathfrak{m}_s$ are pairwise orthogonal with respect to $(\cdot,\cdot),$ then the set $\mathcal{B}:=\tilde{\mathcal{B}}_1\cup\cdots\cup\tilde{\mathcal{B}}_S$ serves as a  $(\cdot,\cdot)$-orthonormal basis for $\mathfrak{m}.$ A basis $\mathcal{B}$ of constructed in this manner will be referred to as a {\it $(\cdot,\cdot)$-orthonormal basis of $\mathfrak{m}$ adapted to the decomposition} $\mathfrak{m}=M_1\oplus\cdots\oplus M_S.$
    \end{remark}
The following theorem is now introduced, presenting a matrix-based characterization of invariant metrics on a homogeneous space of a compact Lie group.

\begin{theorem}\label{main:1}
    Let $G$ be a compact Lie group, $\mathfrak{g}$ the corresponding Lie algebra equipped with an $\Ad(G)$-invariant inner product $(\cdot,\cdot),$ $M=G/H$ a homogeneous $G$-space, $\mathfrak{h}$ the Lie algebra of $H$ and $\mathfrak{m}$ the $(\cdot,\cdot)$-orthogonal complement of $\mathfrak{h}$ in $\mathfrak{g}.$ Let \begin{equation}\label{istropy:decomposition:2}\mathfrak{m}=\mathfrak{m}_1\oplus\cdots\oplus\mathfrak{m}_s
   \end{equation} be a decomposition of $\mathfrak{m}$ into $\Ad(H)$-invariant irreducible pairwise $(\cdot,\cdot)$-orthogonal subspaces, let \begin{equation}\label{isotypical:decomposition}
       \mathfrak{m}=M_1\oplus\cdots\oplus M_S
   \end{equation} be a decomposition of $\mathfrak{m}$ into isotypical summands of \eqref{istropy:decomposition:2} and let $\mathcal{B}$ be a $(\cdot,\cdot)$-orthonormal basis adapted to \eqref{isotypical:decomposition}.
   \begin{itemize}
   \item[$(1)$] For any $G$-invariant metric $g$ on $M$, its associated metric operator $A$ can be written in the basis $\mathcal{B}$ as a block-diagonal matrix of the form:\\
   \begin{equation}\label{matrix:metric:operator}
            \left[A\right]_{\mathcal{B}}=\left(\begin{array}{cccc}
A_1 & 0 & \dots & 0\\
0 & A_2 & \dots & 0\\
\vdots & \vdots & \ddots & \vdots\\
0 & 0 & \dots & A_S\\
\end{array}\right),
        \end{equation}
        where $A_i$ is a square matrix of the following form:
        \begin{equation}\label{matrix:metric:operator:isotypical:summand}
            A_i=\left(\begin{array}{cccc}
\mu_{j_1} \textnormal{I}_{d_i} & \left(A^i_{21}\right)^T & \dots & \left(A^i_{t1}\right)^T\\
A^i_{21} & \mu_{j_2}\textnormal{I}_{d_i} & \dots & \left(A^i_{t2}\right)^T\\
\vdots & \vdots & \ddots & \vdots\\
A^i_{t1} & A^i_{t2} & \dots & \mu_{j_t}\textnormal{I}_{d_i}\\
\end{array}\right),
       \end{equation}
      where $\left(A^i_{lk}\right)^T$ represents the transpose of $A^i_{lk},$ $C_i=\{j_1,...,j_t\},$ $d_i=\dim\mathfrak{m}_{j_1},$ $\mu_{j_1},...,\mu_{j_t}>0$ and $$A^i_{lk}=\left(\begin{array}{cccc}
            \zeta_{j_l}^{j_k} & 0 & \cdots & \\
            0 & \zeta_{j_l}^{j_k} & \cdots & 0\\
            \vdots & \vdots & \ddots & \vdots\\
            0 & 0 & \cdots & \zeta_{j_l}^{j_k}
            \end{array}\right),$$ for some $\zeta_{j_l}^{j_k}$ that is either a real number, the matrix representation of a complex number, or the matrix representation of a quaternion.\\
        
           \item[$(2)$] Conversely, if a matrix is given by \begin{equation}\label{matrix}
                \left(\begin{array}{cccc}
A_1 & 0 & \dots & 0\\
0 & A_2 & \dots & 0\\
\vdots & \vdots & \ddots & \vdots\\
0 & 0 & \dots & A_S\\
\end{array}\right),
\end{equation} where, each $A_i$ has the form of \eqref{matrix:metric:operator:isotypical:summand}, with $$A^i_{lk}\left[\Ad(h)\big{|}_{\mathfrak{m}_{j_1}}\right]_{\mathcal{B}_{j_1}}=\left[\Ad(h)\big{|}_{\mathfrak{m}_{j_1}}\right]_{\mathcal{B}_{j_1}}A^i_{lk},\ l,k\in\{1,...,t\},$$
and this matrix is positive-definite, then the operator $A:\mathfrak{m}\rightarrow\mathfrak{m}$ defined in the basis $\mathcal{B}$ using the formula given in equation \eqref{matrix:metric:operator} corresponds to the metric operator associated with a certain $G$-invariant metric $g.$
    \end{itemize}
    \end{theorem}
    \begin{proof} For the proof of $(1)$ observe that, according to Corollary \ref{j:k:0}, $P_k\circ A\circ i_j=0$ whenever $\mathfrak{m}_j\ncong\mathfrak{m}_k,$ thus $A(M_i)\subseteq M_i,\ i=1,...,S.$ Consequently, \begin{equation*}
        \left[A\right]_{\mathcal{B}}=\left(\begin{array}{cccc}
A_1 & 0 & \dots & 0\\
0 & A_2 & \dots & 0\\
\vdots & \vdots & \ddots & \vdots\\
0 & 0 & \dots & A_S\\
\end{array}\right),
    \end{equation*} where $A_i$ is the matrix of the restriction $A\big{|}_{M_i}:M_i\rightarrow M_i$ in the basis $\tilde{\mathcal{B}}_i.$ Since $\tilde{\mathcal{B}}_i=\mathcal{B}_{j_1}\cup\cdots\cup\mathcal{B}_{j_t},$ then 
    \begin{equation*}
        A_i=\left(\begin{array}{cccc}
\left[P_{j_1}\circ A\circ i_{j_1}\right]_{\mathcal{B}_{j_1}\mathcal{B}_{j_1}} & \left[P_{j_1}\circ A\circ i_{j_2}\right]_{\mathcal{B}_{j_2}\mathcal{B}_{j_1}} & \dots & \left[P_{j_1}\circ A\circ i_{j_t}\right]_{\mathcal{B}_{j_t}\mathcal{B}_{j_1}}\\
\left[P_{j_2}\circ A\circ i_{j_1}\right]_{\mathcal{B}_{j_1}\mathcal{B}_{j_2}} & \left[P_{j_2}\circ A\circ i_{j_2}\right]_{\mathcal{B}_{j_2}\mathcal{B}_{j_2}} & \dots & \left[P_{j_2}\circ A\circ i_{j_t}\right]_{\mathcal{B}_{j_t}\mathcal{B}_{j_2}}\\
\vdots & \vdots & \ddots & \vdots\\
\left[P_{j_t}\circ A\circ i_{j_1}\right]_{\mathcal{B}_{j_1}\mathcal{B}_{j_t}} & \left[P_{j_t}\circ A\circ i_{j_2}\right]_{\mathcal{B}_{j_2}\mathcal{B}_{j_t}} & \dots & \left[P_{j_t}\circ A\circ i_{j_t}\right]_{\mathcal{B}_{j_t}\mathcal{B}_{j_t}}\\
\end{array}\right).
    \end{equation*} By Lemma \ref{almost:main}, for all $l,k\in\{1,...,t\},$ there exists $\zeta_{j_l}^{j_k},$ that is either a real number, the matrix representation of a complex number, or the matrix representation of a quaternion, such that
    $$A^i_{lk}:=\left[P_{j_k}\circ A\circ i_{j_l}\right]_{\mathcal{B}_{j_l}\mathcal{B}_{j_k}}=\left(\begin{array}{cccc}
            \zeta_{j_l}^{j_k} & 0 & \cdots & \\
            0 & \zeta_{j_l}^{j_k} & \cdots & 0\\
            \vdots & \vdots & \ddots & \vdots\\
            0 & 0 & \cdots & \zeta_{j_l}^{j_k}
            \end{array}\right).$$
    Since $A$ is self-adjoint with respect to $(\cdot,\cdot)$ and $\mathcal{B}$ is $(\cdot,\cdot)$-orthonormal, then $\left[A\right]_{\mathcal{B}}$ is symmetric, so $(A_i)^T=A_i,$ and consequently, $$A^i_{lk}:=\left[P_{j_k}\circ A\circ i_{j_l}\right]_{\mathcal{B}_{j_l}\mathcal{B}_{j_k}}=\left[P_{j_l}\circ A\circ i_{j_k}\right]^T_{\mathcal{B}_{j_k}\mathcal{B}_{j_l}}=\left(A_{kl}^i\right)^T.$$ In particular, $A_{ll}^i$ is symmetric, so $\mu_{j_l}:=\zeta_{j_l}^{j_l}$ must be a real number. To see that it is positive, observe that for any  $X\in\mathfrak{m}_{j_l}-\{0\},$ we have that $Q_{j_l}(AX)\in\bigoplus\limits_{r\neq j_l}\mathfrak{m}_r,$ so
    \begin{align*}
        (AX,X)=(P_{j_l}(AX)+Q_{j_l}(AX),X)=(P_{j_l}(AX),X).
        \end{align*}
    Additionally, $i_{j_l}(X)=X$, so
    $$(AX,X)=(P_{j_l}(AX),X)=\left(\left[P_{j_l}\circ A\circ i_{j_l}\right](X),X\right)=(\mu_{j_l}X,X)=\mu_{j_l}\vert\vert X\vert\vert^2.$$
    Since $A$ is positive-definite, then $$\mu_{j_l}\vert\vert X\vert\vert^2=(AX,X)>0.$$ Hence, $\mu_{j_l}>0.$ We have proven $(1).$\\

    For the proof of $(2)$, assume that the matrix \eqref{matrix} is positive-definite. By virtue of Proposition \ref{properties:metric:operator}, it is enough to show that the operator $A:\mathfrak{m}\rightarrow\mathfrak{m}$ defined by \eqref{matrix:metric:operator} is self-adjoint with respect to $(\cdot,\cdot)$ and commutes with $\Ad(h)\big{|}_{\mathfrak{m}}$ for all $h\in H.$ Observe that the basis $\mathcal{B}$ is orthonormal with respect to $(\cdot,\cdot)$ and that the matrix of $A$ in this basis is symmetric, therefore $A$ is self-adjoint. To prove that $A$ commutes with $\Ad(h)\big{|}_{\mathfrak{m}}$ for all $h\in H,$ we may write $\Ad(h)\big{|}_{\mathfrak{m}}$ in the basis $\mathcal{B}.$ Let $h\in H$ and assume that $\mathcal{B}_{j_1}=\{X_1,...,X_{d_i}\}$ so that $\mathcal{B}_{j_k}=\left\{T_{j_1}^{j_k}(X_1),...,T_{j_1}^{j_k}(X_{d_i})\right\}$ for every $k\in\{1,...,t\},$ where $T_{j_1}^{j_k}$ satisfies the conditions $i)$--$iii)$ of Lemma \ref{almost:main}. In particular, \begin{equation}\label{auxiliar:T:intertwining}
        T_{j_1}^{j_k}\circ\Ad(h)\big{|}_{\mathfrak{m}_{j_1}}=\Ad(h)\big{|}_{\mathfrak{m}_{j_k}}\circ T_{j_1}^{j_k}.    
    \end{equation}
    Since $\mathfrak{m}_{j_1}$ is $\Ad(h)$-invariant, then for each $r\in\{1,...,d_i\},$ there exits $\alpha_{1r},...,\alpha_{d_ir}\in\mathbb{R}$ such that \begin{equation}\label{matrix:m_{j_1}}\Ad(h)\big{|}_{\mathfrak{m}_{j_1}}\left(X_r\right)=\alpha_{1r}X_1+\cdots+\alpha_{d_ir}X_{d_i}.\end{equation} Applying $T_{j_1}^{j_k}$ in both sides of this equality obtain \begin{align*}T_{j_1}^{j_k}\left(\Ad(h)\big{|}_{\mathfrak{m}_{j_1}}\right)(X_r)=&T_{j_1}^{j_k}\left(\alpha_{1r}X_1+\cdots+\alpha_{d_ir}X_{d_i}\right)\\
    =&\alpha_{1r}T_{j_1}^{j_k}\left(X_1\right)+\cdots+\alpha_{d_ir}T_{j_1}^{j_k}\left(X_{d_i}\right),
    \end{align*}
    which, by \eqref{auxiliar:T:intertwining}, implies that 
    \begin{equation}\label{matrix:m_{j_k}}
        \Ad(h)\big{|}_{\mathfrak{m}_{j_k}}\left(T_{j_1}^{j_k}\left(X_r\right)\right)=\alpha_{1r}T_{j_1}^{j_k}\left(X_1\right)+\cdots+\alpha_{dr}T_{j_1}^{j_k}\left(X_d\right).
    \end{equation}
    The equalities \eqref{matrix:m_{j_1}} and \eqref{matrix:m_{j_k}} give us
    \begin{equation*}
        \left[\Ad(h)\big{|}_{\mathfrak{m}_{j_1}}\right]_{\mathcal{B}_{j_1}}=\left[\Ad(h)\big{|}_{\mathfrak{m}_{j_k}}\right]_{\mathcal{B}_{j_k}}.
    \end{equation*}
    By defining $D_i:=\left[\Ad(h)\big{|}_{\mathfrak{m}_{j_1}}\right]_{\mathcal{B}_{j_1}}$ we have that the restriction $\Ad(h)\big{|}_{M_i}:M_i\rightarrow M_i$ is expressed in the basis $\tilde{\mathcal{B}}_i=\mathcal{B}_{j_1}\cup\cdots\cup\mathcal{B}_{j_t}$ as \begin{equation}\label{D_i}
        \left[\Ad(h)\big{|}_{M_i}\right]_{\tilde{\mathcal{B}}_i}=\left(\begin{array}{cccc} D_i & 0 & \cdots & 0 \\ 0 & D_i &\cdots & 0 \\ \vdots & \vdots & \ddots & \vdots \\ 0 & 0 & \cdots & D_i \end{array}\right).
    \end{equation}
    Note that $A^i_{lk}=a\textnormal{I}_{d_i}+B^i_{lk}$ where $B_{lk}^i$ is a skew-symmetric matrix. Since $A^i_{lk}$ commutes with $\left[\Ad(h)\big{|}_{\mathfrak{m}_{j_1}}\right]_{\mathcal{B}_{j_1}}=D_i,$ then \begin{align*}
        \left(a\textnormal{I}_{d_i}+B_{lk}^i\right)D_i=D_i\left(a\textnormal{I}_{d_i}+B_{lk}^i\right) \Longrightarrow\ & B_{lk}^{i}D_i=D_iB_{lk}^{i}\\
        \Longrightarrow\ & \left(-B_{lk}^{i}\right)D_i=D_i\left(-B_{lk}^{i}\right)\\
        \Longrightarrow\ & \left(B_{lk}^{i}\right)^TD_i=D_i\left(B_{lk}^{i}\right)^T\\
        \Longrightarrow\ & aD_i+\left(B_{lk}^{i}\right)^TD_i=D_i\left(B_{lk}^{i}\right)^T+aD_i\\
        \Longrightarrow\ & \left(a\textnormal{I}+\left(B_{lk}^{i}\right)^T\right)D_i=\left(\left(B_{lk}^{i}\right)^T+a\textnormal{I}\right)D_i\\
        \Longrightarrow\ & \left(a\textnormal{I}+B_{lk}^{i}\right)^TD_i=\left(B_{lk}^{i}+a\textnormal{I}\right)^TD_i\\
        \Longrightarrow\ & \left(A_{lk}^{i}\right)^TD_i=D_i\left(A_{lk}^{i}\right)^T.\\
    \end{align*} Hence \begin{align*}
        A_i\left[\Ad(h)\big{|}_{M_i}\right]_{\tilde{\mathcal{B}}_i}=&\left(\begin{array}{cccc}
\mu_{j_1} \textnormal{I}_{d_i} & \left(A^i_{21}\right)^T & \dots & \left(A^i_{t1}\right)^T\\
A^i_{21} & \mu_{j_2}\textnormal{I}_{d_i} & \dots & \left(A^i_{t2}\right)^T\\
\vdots & \vdots & \ddots & \vdots\\
A^i_{t1} & A^i_{t2} & \dots & \mu_{j_t}\textnormal{I}_{d_i}\\
\end{array}\right)\left(\begin{array}{cccc} D_i & 0 & \cdots & 0 \\ 0 & D_i &\cdots & 0 \\ \vdots & \vdots & \ddots & \vdots \\ 0 & 0 & \cdots & D_i \end{array}\right)\\
    \\
    =&\left(\begin{array}{cccc}
\mu_{j_1}D_i& \left(A^i_{21}\right)^TD_i & \dots & \left(A^i_{t1}\right)^TD_i\\
A^i_{21}D_i & \mu_{j_2}D_i & \dots & \left(A^i_{t2}\right)^TD_i\\
\vdots & \vdots & \ddots & \vdots\\
A^i_{t1}D_i & A^i_{t2}D_i & \dots & \mu_{j_t}D_i\\
\end{array}\right)\\
\\
    =&\left(\begin{array}{cccc}
\mu_{j_1}D_i& D_i\left(A^i_{21}\right)^T & \dots & D_i\left(A^i_{t1}\right)^T\\
D_iA^i_{21} & \mu_{j_2}D_i & \dots & D_i\left(A^i_{t2}\right)^T\\
\vdots & \vdots & \ddots & \vdots\\
D_iA^i_{t1} & D_iA^i_{t2} & \dots & \mu_{j_t}D_i\\
\end{array}\right)\\
\\
=&\left(\begin{array}{cccc} D_i & 0 & \cdots & 0 \\ 0 & D_i &\cdots & 0 \\ \vdots & \vdots & \ddots & \vdots \\ 0 & 0 & \cdots & D_i 
\end{array}\right)\left(\begin{array}{cccc}
\mu_{j_1} \textnormal{I}_{d_i} & \left(A^i_{21}\right)^T & \dots & \left(A^i_{t1}\right)^T\\
A^i_{21} & \mu_{j_2}\textnormal{I}_{d_i} & \dots & \left(A^i_{t2}\right)^T\\
\vdots & \vdots & \ddots & \vdots\\
A^i_{t1} & A^i_{t2} & \dots & \mu_{j_t}\textnormal{I}_{d_i}\\
\end{array}\right)\\
\\
=& \left[\Ad(h)\big{|}_{M_i}\right]_{\tilde{\mathcal{B}}_i}A_i.
    \end{align*}
    This holds for all $i\in\{1,...,S\},$ so
    \begin{align*}
        \left[A\right]_{\mathcal{B}}\left[\Ad(h)\big{|}_{\mathfrak{m}}\right]_{\mathcal{B}}=&\left(\begin{array}{ccc} A_1 & \cdots & 0 \\ \vdots & \ddots & \vdots \\ 0 & \cdots & A_S 
    \end{array}\right)\left(\begin{array}{ccc} \left[\Ad(h)\big{|}_{M_1}\right]_{\tilde{\mathcal{B}}_1} & \cdots & 0 \\ \vdots & \ddots & \vdots \\ 0 & \cdots & \left[\Ad(h)\big{|}_{M_S}\right]_{\tilde{\mathcal{B}}_S} 
    \end{array}\right)\\
    \\
    =&\left(\begin{array}{ccc} A_1\left[\Ad(h)\big{|}_{M_1}\right]_{\tilde{\mathcal{B}}_1} & \cdots & 0 \\ \vdots & \ddots & \vdots \\ 0 & \cdots & A_S\left[\Ad(h)\big{|}_{M_S}\right]_{\tilde{\mathcal{B}}_S} 
    \end{array}\right)\\
    \\
    =&\left(\begin{array}{ccc} \left[\Ad(h)\big{|}_{M_1}\right]_{\tilde{\mathcal{B}}_1}A_1 & \cdots & 0 \\ \vdots & \ddots & \vdots \\ 0 & \cdots & \left[\Ad(h)\big{|}_{M_S}\right]_{\tilde{\mathcal{B}}_S}A_S 
    \end{array}\right)\\
    \\
    =&\left(\begin{array}{ccc} \left[\Ad(h)\big{|}_{M_1}\right]_{\tilde{\mathcal{B}}_1} & \cdots & 0 \\ \vdots & \ddots & \vdots \\ 0 & \cdots & \left[\Ad(h)\big{|}_{M_S}\right]_{\tilde{\mathcal{B}}_S} 
    \end{array}\right)\left(\begin{array}{ccc} A_1 & \cdots & 0 \\ \vdots & \ddots & \vdots \\ 0 & \cdots & A_S 
    \end{array}\right)\\
    \\
    =&\left[\Ad(h)\big{|}_{\mathfrak{m}}\right]_{\mathcal{B}}\left[A\right]_{\mathcal{B}}.
    \end{align*}
    This shows that $A$ commutes with $\Ad(h)\big{|}_{\mathfrak{m}}.$ The proof is complete.
    \end{proof}
    \begin{remark}
        Let us introduce some useful notations.  Consider two indices $p,q\in\{1,...,s\}.$ If $\mathfrak{m}_p$ is not equivalent to $\mathfrak{m}_q,$ we define $T_p^q:\mathfrak{m}_p\rightarrow\mathfrak{m}_q$ as the zero map. On the other hand, if there exists $i\in\{1,...,S\}$ such that both $p$ and $q$ belong to $C_i=\{j_1,...,j_t\},$ there exist $l,k\in\{1,...,t\}$ such that $p=j_l$ and $q=j_k.$ In this case, we define $T_p^q:=T_{j_l}^{j_k}=T_{j_1}^{j_k}\circ \left(T_{j_1}^{j_l}\right)^{-1},$ as explained in the proof of Lemma \ref{almost:main}. Observe that $T_q^p=(T_p^q)^{-1}.$ To summarize, we have a family $\{T_p^q:p,q\in\{1,...,s\}\}$ of linear operators that satisfies the following properties:
    \begin{itemize}
        \item[$i)$] $T_p^p=\textnormal{I}_{\mathfrak{m}_p},$ $p=1,...,s.$\\
        
        \item[$ii)$] $T_p^q=0$ whenever $\mathfrak{m}_p\ncong\mathfrak{m}_q.$\\
        
        \item[$iii)$] If $\mathfrak{m}_p\cong\mathfrak{m}_q,$ then $T_p^q:\mathfrak{m}_p\rightarrow\mathfrak{m}_q$ is a linear isomorphism that preserves $(\cdot,\cdot)$ and $T_p^q=(T_p^q)^{-1}.$
    \end{itemize}
    \end{remark}
    The following corollary is a particular case of Theorem \ref{main:1}. Nevertheless, we are presenting it separately, as it will be instrumental in Section \ref{section:3} for establishing the main result of this paper\begin{corollary}\label{instrumental:corollary} Let $\mu_1,...,\mu_s>0.$ Then, there exist $\epsilon>0$ such that for every family $\{a_{pq}:p,q\in\{1,...,s\},\ p\neq q\}$ of real numbers satisfying $|a_{pq}|<\epsilon$ and $a_{pq}=a_{qp},$ the linear map $A:\mathfrak{m}\rightarrow \mathfrak{m}$ defined as
    \begin{equation}\label{metric:operator:epsilon}A(X)=\mu_pX+\sum\limits_{q\neq p}a_{pq}T_p^q(X),\ X\in\mathfrak{m}_p,\ p=1,...,s;\end{equation} is the metric operator associated with a $G$-invariant metric on $M.$
    \end{corollary}
    \begin{proof} Let $\mathcal{B}$ be an $(\cdot,\cdot)$-orthonormal basis of $\mathfrak{m}$ adapted to \eqref{isotypical:decomposition}, and constructed from the family $\{T_p^q:p,q\in\{1,...,s\}\}.$ For any family $\{a_{pq}:p,q\in\{1,...,s\},\ p\neq q\}$ of real numbers satisfying $a_{pq}=a_{qp},$ the linear map \eqref{metric:operator:epsilon} can be written in a  rearrangement of the basis $\mathcal{B},$ denoted as $\tilde{\mathcal{B}},$ as follows:
    \begin{equation}\label{positive:matrix}
        \left[A\right]_{\tilde{\mathcal{B}}}=\left(\begin{array}{ccccc}\mu_1\textnormal{I}_{d_1} & A_{21} & A_{31} & \cdots & A_{s1} \\
        A_{21} & \mu_2\textnormal{I}_{d_2} & A_{32} & \cdots & A_{s2}\\
        A_{31} & A_{32} & \mu_3\textnormal{I}_{d_3} & \cdots & A_{s3}\\
        \vdots & \vdots & \vdots & \ddots & \vdots\\
        A_{s1} & A_{s2} & A_{s3} & \cdots &\mu_{s}\textnormal{I}_{d_s}
        \end{array}\right),
    \end{equation}
    where $d_p=\dim\mathfrak{m}_p,\ p=1,...,s;$ and $$A_{pq}=\left\{\begin{array}{ll}0, & \text{if}\ \mathfrak{m}_p\ncong\mathfrak{m}_q\\ a_{pq}\textnormal{I}_{d_p}, & \text{if}\ \mathfrak{m}_p\cong\mathfrak{m}_q.\end{array}\right.$$ Since $A_{pq}$ is a scalar multiple of the the identity matrix when $\mathfrak{m}_p\cong\mathfrak{m}_q,$ it follows that $A_{pq}$ commutes with any matrix of its same order. Hence, according to Theorem \ref{main:1} $(2),$ $A$ will be the metric operator associated with a $G$-invariant metric whenever the matrix in \eqref{positive:matrix} is positive-definite. We will show that this condition holds when the values of $a_{pq}$ are sufficiently small. In fact, consider the polynomial $P:\mathbb{C}\longrightarrow\mathbb{C}$ defined as $$P(z)=(z-\mu_1)^{d_1}\cdots(z-\mu_{s})^{d_s}=z^d+\alpha_1z^{d-1}+\cdots+\alpha_{d-1}z+\alpha_d$$ where $d=d_1+\cdots+d_s,$ and let $$\eta=\min\left\{\frac{\left|\mu_p-\mu_q\right|}{2}:p\neq q\right\}\cup\{\mu_p:p=1,...,2\}.$$ According to \cite[Theorem B]{HM}, there exists $\delta>0$ such that if $|\beta_j-\alpha_j|<\delta,$ for $j=1,...,d;$ then the polynomial $$Q(z)=z^d+\beta_1z^{d-1}+\cdots+\beta_{d-1}z+\beta_d$$
    has exactly $d_p$ roots in $B(\mu_p,\eta).$ If $\lambda$ is a root of $Q(z),$ there exists $p\in\{1,...,s\}$ such that $|\lambda-\mu_p|<\eta\leq\mu_p.$ Thus
    \begin{align*}
        |\textnormal{Re}(\lambda)-\mu_p|\leq|\lambda-\mu_p|<\mu_p,
    \end{align*}
    which means that $\textnormal{Re}(\lambda)\in(0,2\mu_p).$ In particular, $\textnormal{Re}(\lambda)>0.$\\
    
    Now, consider the polynomial $$Q_A(z)=\det\left(z\textnormal{I}_d-\left[A\right]_{\tilde{\mathcal{B}}}\right)=z^d+f_1(\textbf{a})z^{d-1}+\cdots+f_{d-1}(\textbf{a})z+f_d(\textbf{a}),$$ where each $f_j(\textbf{a})$ is a polynomial function of $\textbf{a}:=(a_{pq})_{1\leq q<p\leq s}.$ Observe that $f_j(\textbf{0})=\alpha_j.$ The continuity of $f_j$ implies that there exists $\epsilon_j>0$ such that $$|a_{pq}|<\epsilon_j,\ p> q\Longrightarrow \left|f_j(\textbf{a})-\alpha_j\right|<\delta.$$ Let $\epsilon=\min\{\epsilon_1,...,\epsilon_d\}.$ Then we have
    \begin{align*}
        |a_{pq}|<\epsilon,\ p> q\Longrightarrow\left|f_j(\textbf{a})-\alpha_j\right|<\delta,\ j=1,...,d.
    \end{align*}
    This implies that the real part of all $Q_A(z)$ roots is positive. In other words, the real part of all eigenvalues of $\left[A\right]_{\tilde{\mathcal{B}}}$ is positive. Since $\left[A\right]_{\tilde{\mathcal{B}}}$ is symmetric, all its eigenvalues are real numbers. Therefore, all eigenvalues of $\left[A\right]_{\tilde{\mathcal{B}}}$ are positive whenever $|a_{pq}|<\epsilon$ for $p> q$ or, equivalently, $\left[A\right]_{\tilde{\mathcal{B}}}$ is positive-definite when $|a_{pq}|<\epsilon$ for $p> q.$ The proof is complete.
    \end{proof}

    \section{Equigeodesic Vectors}\label{section:3}
    \begin{definition}
        Let $M=G/H$ be a homogeneous space endowed with a $G$-invariant metric $g.$ A vector $X\in\mathfrak{g}$ is called a {\it geodesic} with respect to $g,$ if the orbit of its 1-parameter subgroup $\exp(tX)\cdot o$ is a geodesic in $(M,g).$ If $X\in\mathfrak{g}$ is a geodesic vector with respect to any $G$-invariant metric, it is then referred to as an {\it equigeodesic vector.}      
    \end{definition}
    Given $X\in\mathfrak{g},$ we denote $X_{\mathfrak{h}}$ and $X_{\mathfrak{m}}$ the orthogonal projections of $X$ on the subspaces $\mathfrak{h}$ and $\mathfrak{m}$ respectively, so that, $$X=X_{\mathfrak{h}}+X_{\mathfrak{m}},\ X_{\mathfrak{h}}\in\mathfrak{h},\ X_{\mathfrak{m}}\in\mathfrak{m}.$$ The following theorem provides a necessary and sufficient condition for a vector $X\in\mathfrak{g}$ to be geodesic with respect to a $G$-invariant metric $g.$ It was established by Kowalski and Vanhecke in \cite{KV}.
    \begin{theorem}\label{KV:condition}
        A vector $X\in\mathfrak{g}$ is geodesic with respect to a $G$-invariant metric $g$ if and only if \begin{equation}\label{KV:formula}
            \left\langle\left[X,Y\right]_{\mathfrak{m}},X_{\mathfrak{m}}\right\rangle_g=0,
        \end{equation}
        for all $Y\in\mathfrak{m}.$
    \end{theorem}
    \begin{corollary}\label{CGN:condition}
        A vector $X\in\mathfrak{g}$ is geodesic with respect to a $G$-invariant metric $g$ if and only if \begin{equation*}
            \left[X,AX_{\mathfrak{m}}\right]_{\mathfrak{m}}=0,
        \end{equation*}
        where $A$ is the metric operator associated with $g.$
    \end{corollary}
    \begin{proof} Let $X\in\mathfrak{g}.$ For every $Y\in\mathfrak{m}$ we have
    \begin{align*}
        \left\langle\left[X,Y\right]_{\mathfrak{m}},X_{\mathfrak{m}}\right\rangle_g = &\left\langle X_{\mathfrak{m}},\left[X,Y\right]_{\mathfrak{m}}\right\rangle_g\\
        = & \left(AX_{\mathfrak{m}},\left[X,Y\right]_{\mathfrak{m}}\right)\\
        = & \left(AX_{\mathfrak{m}},\left[X,Y\right]\right),\hspace{1.2cm} \text{since}\ \mathfrak{m}\ \text{and}\ \mathfrak{h}\ \text{are}\ (\cdot,\cdot)\text{-orthogonal}\\
        = & -\left(\left[X,AX_{\mathfrak{m}}\right],Y\right),\hspace{0.7cm} \text{since}\ (\cdot,\cdot)\ \text{is}\ \Ad(G)\text{-invariant}\\
        = & -\left(\left[X,AX_{\mathfrak{m}}\right]_{\mathfrak{m}},Y\right),\hspace{0.45cm} \text{since}\ \mathfrak{m}\ \text{and}\ \mathfrak{h}\ \text{are}\ (\cdot,\cdot)\text{-orthogonal.}
    \end{align*}
    Thus, the condition \eqref{KV:formula} is satisfied by $X$ if and only if $$\left(\left[X,AX_{\mathfrak{m}}\right]_{\mathfrak{m}},Y\right)=0,$$ for all $Y\in\mathfrak{m}.$ Since $(\cdot,\cdot)\big{|}_{\mathfrak{m}\times\mathfrak{m}}$ is an inner product, this holds if and only if $$\left[X,AX_{\mathfrak{m}}\right]_{\mathfrak{m}}=0.$$
    \end{proof}
    We will apply the characterization introduced in the previous section of $G$-invariant metrics on a homogeneous space to establish the following theorem, which provides a strong necessary condition for a vector $X\in\mathfrak{m}$ to be equigeodesic. In the subsequent section, we will see its utility in characterizing equigeodesic vectors on a homogeneous space with equivalent isotropy summands.
    \begin{theorem}\label{main:2}
        If $X\in\mathfrak{m}$ is an equigeodesic vector then \begin{equation}\label{main:formula}
            \left[X,T_i^j(P_i(X))+T_j^i(P_j(X))\right]_{\mathfrak{m}}=0,\ i,j=1,...,s.
        \end{equation}
    \end{theorem}
    \begin{proof} Let $X\in\mathfrak{m}$ be an equigeodesic vector. Then $$\left[X,AX\right]_{\mathfrak{m}}=0,$$ for every linear map $A:\mathfrak{m}\rightarrow\mathfrak{m}$ which is the metric operator associated with some $G$-invariant metric. For $\mu_1,...,\mu_s>0$ and a family $\{a_{pq}:p,q\in\{1,...,s\},\ p\neq q\}$ of real numbers such that $a_{pq}=a_{qp},$ consider (as in Corollary \ref{instrumental:corollary}) the linear map $A:\mathfrak{m}\rightarrow\mathfrak{m}$ defined as \begin{equation}\label{auxiliar:0}A(Y)=\mu_pY+\sum\limits_{q\neq p}a_{pq}T_p^q(Y),\ Y\in\mathfrak{m}_p,\ p=1,...,s.\end{equation}
    Computing $\left[X,AX\right]_{\mathfrak{m}}$ we obtain: 
    \begin{align*}
        \left[X,AX\right]_{\mathfrak{m}}=&\left[X,A\left(\sum\limits_{p=1}^{s}P_p(X)\right)\right]_{\mathfrak{m}}\\
        =&\left[X,\sum\limits_{p=1}^{s}A\left(P_p(X)\right)\right]_{\mathfrak{m}}\\ 
        =& \left[X,\sum\limits_{p=1}^{s}\left\{\mu_pP_p(X)+\sum\limits_{q\neq p}a_{pq}T_p^q(P_p(X))\right\}\right]_{\mathfrak{m}}\\
        =& \sum\limits_{p=1}^s\mu_p\left[X,P_p(X)\right]_{\mathfrak{m}}+\sum\limits_{q\neq p}a_{pq}\left[X,T_p^q(P_p(X))\right]_{\mathfrak{m}}\\
        =& \sum\limits_{p=1}^s\mu_p\left[\sum\limits_{q=1}^sP_q(X),P_p(X)\right]_{\mathfrak{m}}+\sum\limits_{q\neq p}a_{pq}\left[X,T_p^q(P_p(X))\right]_{\mathfrak{m}}\\
        =& \sum\limits_{p=1}^s\sum\limits_{q=1}^s\mu_p\left[P_q(X),P_p(X)\right]_{\mathfrak{m}}+\sum\limits_{q\neq p}a_{pq}\left[X,T_p^q(P_p(X))\right]_{\mathfrak{m}}\\
        =& \sum\limits_{p\neq q}\mu_p\left[P_q(X),P_p(X)\right]_{\mathfrak{m}}+\sum\limits_{q\neq p}a_{pq}\left[X,T_p^q(P_p(X))\right]_{\mathfrak{m}}\\
        =& \sum\limits_{p>q}\mu_p\left[P_q(X),P_p(X)\right]_{\mathfrak{m}}+\sum\limits_{q>p}\mu_p\left[P_q(X),P_p(X)\right]_{\mathfrak{m}}\\
         &+\sum\limits_{p>q}a_{pq}\left[X,T_p^q(P_p(X))\right]_{\mathfrak{m}}+\sum\limits_{q>p}a_{pq}\left[X,T_p^q(P_p(X))\right]_{\mathfrak{m}}\\
         =& \sum\limits_{p>q}\mu_p\left[P_q(X),P_p(X)\right]_{\mathfrak{m}}+\sum\limits_{p>q}\mu_q\left[P_p(X),P_q(X)\right]_{\mathfrak{m}}\\
         &+\sum\limits_{p>q}a_{pq}\left[X,T_p^q(P_p(X))\right]_{\mathfrak{m}}+\sum\limits_{p>q}a_{qp}\left[X,T_q^p(P_q(X))\right]_{\mathfrak{m}}\\
         =& -\sum\limits_{p>q}\mu_p\left[P_p(X),P_q(X)\right]_{\mathfrak{m}}+\sum\limits_{p>q}\mu_q\left[P_p(X),P_q(X)\right]_{\mathfrak{m}}\\
         &+\sum\limits_{p>q}a_{pq}\left[X,T_p^q(P_p(X))\right]_{\mathfrak{m}}+\sum\limits_{p>q}a_{pq}\left[X,T_q^p(P_q(X))\right]_{\mathfrak{m}}\\
         =&\sum\limits_{p>q}\left(\mu_q-\mu_p\right)\left[P_p(X),P_q(X)\right]_{\mathfrak{m}}+\sum\limits_{p>q}a_{pq}\left[X,T_p^q(P_p(X))+T_q^p(P_q(X))\right]_{\mathfrak{m}}.
    \end{align*} Fix $i\in\{1,...,s\}$ and let $$\mu_p=\left\{\begin{array}{ll}2,&\textnormal{if}\ p=i,\\ 1, &\textnormal{if}\ p\neq i,\end{array}\right. a_{pq}=0,\ p\neq q.$$ Then, according to Corollary \ref{instrumental:corollary}, the formula \eqref{auxiliar:0} defines a metric operator associated with certain $G$-invariant metric. Hence,
    \begin{align*}
        0=& \sum\limits_{p>q}\left(\mu_q-\mu_p\right)\left[P_p(X),P_q(X)\right]_{\mathfrak{m}}+\sum\limits_{p>q}a_{pq}\left[X,T_p^q(P_p(X))+T_q^p(P_q(X))\right]_{\mathfrak{m}}\\
        =& \sum\limits_{p>i}\left(2-1\right)\left[P_p(X),P_i(X)\right]_{\mathfrak{m}}+\sum\limits_{i>p}\left(1-2\right)\left[P_i(X),P_p(X)\right]_{\mathfrak{m}}\\
        =& \sum\limits_{p>i}\left[P_p(X),P_i(X)\right]_{\mathfrak{m}}+\sum\limits_{i>p}\left[P_p(X),P_i(X)\right]_{\mathfrak{m}}\\
        =& \sum\limits_{p\neq i}\left[P_p(X),P_i(X)\right]_{\mathfrak{m}}\\
        =& \sum\limits_{p\neq i}\left[P_p(X),P_i(X)\right]_{\mathfrak{m}}+\left[P_i(X),P_i(X)\right]_{\mathfrak{m}}\\
        =&\left[\sum\limits_{p=1}^sP_p(X),P_i(X)\right]_{\mathfrak{m}}\\
        =&\left[X,P_i(X)\right]_{\mathfrak{m}}\\
        =&\left[X,T_i^i(P_i(X))\right]_{\mathfrak{m}}.
    \end{align*}
    This proves the result for $i=j.$ Now, consider $\mu_p=1,\ p=1,...,s.$ By Corollary \ref{instrumental:corollary}, there exists $\epsilon>0$ such that for every family $\{a_{pq}:p,q\in\{1,...,s\},\ p\neq q\}\subseteq (-\epsilon,\epsilon)$ satisfying $a_{pq}=a_{qp},$ the linear map \eqref{auxiliar:0} is the metric operator associated with some $G$-invariant metric. Given $i\neq j,$ without loss of generality suppose that $i>j$ and set
    $$a_{pq}=a_{qp}=\left\{\begin{array}{ll}
       \displaystyle\frac{\epsilon}{2},  & \text{if}\ \{p,q\}=\{i,j\}, \\
       \\
        0, &\text{otherwise.} 
    \end{array}\right.$$
    Since $[X,AX]_{\mathfrak{m}}=0,$ then 
    \begin{align*}
        0=& \sum\limits_{p>q}\left(\mu_q-\mu_p\right)\left[P_p(X),P_q(X)\right]_{\mathfrak{m}}+\sum\limits_{p>q}a_{pq}\left[X,T_p^q(P_p(X))+T_q^p(P_q(X))\right]_{\mathfrak{m}}\\
        =& \sum\limits_{p>q}\left(1-1\right)\left[P_p(X),P_q(X)\right]_{\mathfrak{m}}+\frac{\epsilon}{2}\left[X,T_i^j(P_i(X))+T_j^i(P_j(X))\right]_{\mathfrak{m}}\\
        =&\frac{\epsilon}{2}\left[X,T_i^j(P_i(X))+T_j^i(P_j(X))\right]_{\mathfrak{m}}.
    \end{align*}
    Hence $$\left[X,T_i^j(P_i(X))+T_j^i(P_j(X))\right]_{\mathfrak{m}}=0.$$
    \end{proof}
    It is essential to emphasize that the formula \eqref{main:formula} eliminates the dependence on the metric operator $A.$ Instead, it solely relies on the projections associated with the isotropy summands and the corresponding equivariant maps between them. However, formula \eqref{main:formula} is not generally a sufficient condition for a vector $X\in\mathfrak{m}$ to be equigeodesic. The next example illustrates this.
    \begin{example}\label{not:sufficient} Consider the Stiefel manifold $V_2\left(\mathbb{R}^4\right)=\textnormal{SO}(4)/\textnormal{SO}(2),$ where $\textnormal{SO}(2)$ is identified with the subgroup $$\left\{\left(\begin{array}{cccc}
        u & -v  & 0 & 0\\
        v & u & 0 & 0\\
        0 & 0 & 1 & 0\\
        0 & 0 & 0 & 1
    \end{array}\right):u^2+v^2=1\right\}\subseteq \textnormal{SO}(4).$$ The Lie algebra $\mathfrak{so}(4)$ of $\textnormal{SO}(4)$ consists of all skew-symmetric square matrices of order 4. Set $$(\cdot,\cdot):=\dfrac{1}{4}B\ \text{and}\ E_{ij}:=\left(\delta_{ik}\delta_{jl}\right)_{(k,l)},\ 1\leq i\neq j\leq 4,$$ where $B$ denotes the Killing form of $\mathfrak{so}(4)$ (which is $\Ad($SO(4))-invariant) and $\delta_{mn}$ is the Kronecker delta. This means that $E_{ij}$ is the $4\times 4$ matrix  with a 1 in the $(i,j)$-entry and zeros everywhere else. Define \begin{align*}
        X_1:=E_{21}-E_{12},\ X_2:=E_{43}-E_{34},\ X_3:=E_{31}-E_{13}\\
        X_4:=E_{42}-E_{24},\ X_5:=E_{32}-E_{23},\ X_6:=E_{41}-E_{14}.
    \end{align*}
        An $(\cdot,\cdot)$-orthogonal reductive decomposition is given by $$\mathfrak{so}(4)=\mathfrak{so}(2)\oplus\mathfrak{m}=\mathfrak{so}(2)\oplus\mathfrak{m}_1\oplus\mathfrak{m}_2\oplus\mathfrak{m}_3,$$ where \begin{align*}&\mathfrak{so}(2)=\textnormal{span}\{X_1\},\\
        &\mathfrak{m}_1=\textnormal{span}\{X_2\},\\ &\mathfrak{m}_2=\textnormal{span}\{X_3,X_5\},\\
        &\mathfrak{m}_3=\textnormal{span}\{X_4,X_6\}.
        \end{align*} The isotropy summands $\mathfrak{m}_2$ and $\mathfrak{m}_3$ are equivalent and the map $$\begin{array}{rrcl}T_2^3:&\mathfrak{m}_2&\longrightarrow&\mathfrak{m}_3\\ & X_3 &\longmapsto & X_4\\ & X_5 &\longmapsto & -X_6\end{array}$$ is an intertwining operator preserving $(\cdot,\cdot).$ For a vector $X=\sum\limits_{k=2}^6x_kX_k\in\mathfrak{m}$ we have \begin{align*}
            & \left[X,P_1(X)\right]_{\mathfrak{m}}=  \left[X,x_2X_2\right]_{\mathfrak{m}}=x_2x_6X_3-x_2x_5X_4+x_2x_4X_5-x_2x_3X_6\\
            \\
            &\left[X,P_2(X)\right]_{\mathfrak{m}}= \left[X,x_3X_3+x_5X_5\right]_{\mathfrak{m}} = -\left(x_3x_6+x_4x_5\right)X_2\\
            \\
            &\left[X,P_3(X)\right]_{\mathfrak{m}}=  \left[X,x_4X_4+x_6X_6\right]_{\mathfrak{m}}=\left(x_3x_6+x_4x_5\right)X_2\\
            \\
            &\left[X,T_2^3(P_2(X))+T_3^2(P_3(X))\right]_{\mathfrak{m}} = \left[X,x_4X_3+x_3X_4-x_6X_5-x_5X_6\right]_{\mathfrak{m}}\\
            \\
            &\hspace{5.525cm}=x_2x_5X_3-x_2x_6X_4-x_2x_3X_5+x_2x_4X_6
        \end{align*} Thus, $X$ satisfies the equations \eqref{main:formula} if and only if $$x_2x_k=x_3x_6+x_4x_5=0,\ k=3,4,5,6.$$ On the other hand, it is easy to see that for each $$h=\left(\begin{array}{cccc}
        u & -v  & 0 & 0\\
        v & u & 0 & 0\\
        0 & 0 & 1 & 0\\
        0 & 0 & 0 & 1
    \end{array}\right)\in\textnormal{SO}(2),$$ $\Ad(h)$ is expressed in the basis $\mathcal{B}=\{X_2,X_3,X_5,X_4,-X_6\}$ as 
    $$\left[\Ad(h)\right]_{\mathcal{B}}=\left(\begin{array}{ccccc}
        1 & 0 & 0 & 0 & 0\\
        0 & u & -v & 0 & 0\\
        0 & v & u & 0 & 0\\
        0 & 0 & 0 & u & -v\\
        0 & 0 & 0 & v & u\\
    \end{array}\right).$$ According to Theorem \ref{main:1}, the set of all metric operators associated with $\textnormal{SO}(4)$-invariant metrics on $V_2\left(\mathbb{R}^4\right)$ is the set of all linear maps $A$ that are written in the basis $\mathcal{B}$ as $$\left[A\right]_{\mathcal{B}}=\left(\begin{array}{ccccc}
        \mu_1 & 0 & 0 & 0 & 0\\
        0 & \mu_2 & 0 & a & b\\
        0 & 0 & \mu_2 & -b & a\\
        0 & a & -b & \mu_3 & 0\\
        0 & b & a & 0 & \mu_3\\
    \end{array}\right),\ \mu_1,\mu_2,\mu_3>0\ \text{and}\  a^2+b^2<\mu_2\mu_3.$$ Therefore, for a vector $X=\sum\limits_{k=2}^6x_kX_k$ the expression $[X,AX]_{\mathfrak{m}}$ is:
    \begin{align*}
        \left[X,AX\right]_{\mathfrak{m}}= &\left[\sum\limits_{k=2}^6x_kX_k,\sum\limits_{k=2}^6x_kAX_k\right]_{\mathfrak{m}}\\
        \\
        =&\left\{(\mu_3-\mu_2)(x_3x_6+x_4x_5)-b(x_3^2+x_4^2+x_5^2+x_6^2)\right\}X_2\\
        &+\left\{(\mu_1-\mu_3)x_2x_6+ax_2x_5+bx_2x_3\right\}X_3\\
        &+\left\{(\mu_2-\mu_1)x_2x_5-ax_2x_6+bx_2x_4\right\}X_4\\
        &+\left\{(\mu_1-\mu_3)x_2x_4-ax_2x_3+bx_2x_5\right\}X_5\\
        &+\left\{(\mu_2-\mu_1)x_2x_3+ax_2x_4+bx_2x_6\right\}X_6.
    \end{align*}
    This implies that $\left[X,AX\right]_{\mathfrak{m}}=0$ for every metric operator $A$, if and only if $$x_2x_k=x_3x_6=x_3^2+x_4^2+x_5^2+x_6^2=0,\ k=3,4,5,6;$$ or, simply, $x_3=x_4=x_5=x_6=0.$ In this case, any nonzero vector $X=x_3X_4+x_4X_4+x_5X_5+x_6X_6,$ with $x_3x_6+x_4x_5=0$ is not an equigeodesic vector, but it satisfies the conditions \eqref{main:formula}.
    \end{example}
    \begin{remark} If an  isotropy summand $\mathfrak{m}_j$ has multiplicity equals to one, then every vector $X\in\mathfrak{m}_j$ is equigeodesic. This is because $AX$ is a scalar multiple of $X$ for any metric operator $A,$ and thus, $[X,AX]=0.$ However, this does not hold in general for an isotropy summand of multiplicity greater than one. As shown in the example above, every $X\in\mathfrak{m}_j,\ j=2,3$ is not equigeodesic.
    \end{remark}
    \begin{remark} In the example \ref{not:sufficient}, the formula \eqref{main:formula} fails to be a sufficient condition for a vector $X\in\mathfrak{m}$ to be equigeodesic because the representation $$\Ad^{\textnormal{SO}(2)}\big{|}_{\mathfrak{m}}:\textnormal{SO}(2)\rightarrow\textnormal {GL}(\mathfrak{m}_2)$$ is not orthogonal, this means that the space of all equivariant endomorphism of this representation is not isomorphic to $\mathbb{R}.$ This does not occur in the case of the Stiefel manifold $V_2\left(\mathbb{R}^n\right)$ when $n\neq 4.$ A complete characterization of equigeodesic vectors for $V_2\left(\mathbb{R}^n\right)$ was done by Statha in \cite{S}.
    \end{remark}
    The following theorem establishes a scenario in which formula \eqref{main:formula} serves as a characterization of equigeodesic vectors. Before presenting it, let us recall that, due to Theorem \ref{irreducible:real:representations}, we have that $$\textnormal{End}(\mathfrak{m}_j)\cong \mathbb{R},\ \mathbb{C},\ \textnormal{or}\ \mathbb{H},\ j=1,...,s;$$ where $\textnormal{End}(\mathfrak{m}_j)$ is the set of all equivariant endomorphisms of the representation $\Ad^H\big{|}_{\mathfrak{m}_j}.$    \begin{theorem}\label{main:3}
        Let $M$ be a homogeneous $G$-space and assume that $\textnormal{End}(\mathfrak{m}_j)\cong \mathbb{R}$ for each isotropy summand $\mathfrak{m}_j$ with multiplicity greater than one. Then $X\in\mathfrak{m}$ is equigeodesic if and only if \begin{equation}\label{main:formula:2}
            \left[X,T_i^j(P_i(X))+T_j^i(P_j(X))\right]_{\mathfrak{m}}=0,\ i,j=1,...,s.
        \end{equation}
    \end{theorem}
    \begin{proof} We only need to show that if $X\in\mathfrak{m}$ satisfies the conditions: $$\left[X,T_i^j(P_i(X))+T_j^i(P_j(X))\right]_{\mathfrak{m}}=0,\ i,j=1,...,s;$$ then it is indeed an equigeodesic vector. To do this, we will show that any metric operator $A$, associated with a $G$-invariant metric, takes the form of \eqref{auxiliar:0}. According to Theorem \ref{main:1} $(1)$, any metric operator $A$ has the form $$A(Y)=\mu_pY+\sum\limits_{q\neq p}\left[P_q\circ A\circ i_p\right](Y),\ Y\in\mathfrak{m}_p,\ p=1,...,s;$$ where $\mu_1,...,\mu_s>0.$ If $p\neq q,$ by Corollary \ref{j:k:0}, $P_q\circ A\circ i_p=0$ whenever  $\mathfrak{m}_p\ncong\mathfrak{m}_q.$ In such cases, we can set $a_{pq}=a_{qp}=0.$ However, if $\mathfrak{m}_p\cong\mathfrak{m}_q,$ then $P_q\circ A\circ i_p:\mathfrak{m}_p\rightarrow \mathfrak{m}_q$ is an equivariant map between $\Ad^H\big{|}_{\mathfrak{m}_p}$ and $\Ad^H\big{|}_{\mathfrak{m}_q}$ (see Proposition \ref{A:Equivariant}). Therefore, $T_q^p\circ P_q\circ A\circ i_p\in\textnormal{End}(\mathfrak{m}_p)\cong\mathbb{R}.$ This means that there exists $a_{pq}\in\mathbb{R}$ such that $$T_q^p\circ P_q\circ A\circ i_p=a_{pq}\textnormal{I}_{\mathfrak{m}_p},$$ or, equivalently, $$P_q\circ A\circ i_p=a_{pq}T_p^q$$ (since $(T_q^p)^{-1}=T_p^q$). By applying Theorem \ref{main:1} $(1),$ we can conclude that $a_{pq}=a_{qp}.$ Hence, any metric operator $A$ can be written as $$A(Y)=\mu_pY+\sum\limits_{q\neq p}a_{pq}T_p^q(Y),\ Y\in\mathfrak{m}_p,\ p=1,...,s;$$ where $\mu_1,...,\mu_s>0$ and $\{a_{pq}:p,q\in\{1,...,s\},\ p\neq q\}\subseteq\mathbb{R}$ with $a_{pq}=a_{qp}.$ As previously shown in the proof of Theorem \ref{main:2}, for all $X\in\mathfrak{m}$ we have that $$\left[X,AX\right]_{\mathfrak{m}}=\sum\limits_{p=1}^s\mu_p\left[X,P_p(X)\right]_{\mathfrak{m}}+\sum\limits_{p>q}a_{pq}\left[X,T_p^q(P_p(X))+T_q^p(P_q(X))\right]_{\mathfrak{m}}.$$ Assuming that $X\in\mathfrak{m}$ satisfies  $$\left[X,T_i^j(P_i(X))+T_j^i(P_j(X))\right]_{\mathfrak{m}}=0,\ i,j=1,...,s;$$ we obtain $$[X,P_i(X)]_{\mathfrak{m}}=[X,T_i^{i}P_i(X)]_{\mathfrak{m}}=\frac{1}{2}[X,T_i^{i}P_i(X)+T_i^{i}P_i(X)]_{\mathfrak{m}}=0,\ i=1,...,s.$$ This implies that $\left[X,AX\right]_{\mathfrak{m}}=0$ for any metric operator $A.$ Hence, by Corollary \ref{CGN:condition}, we can conclude that $X$ is equigeodesic.
    \end{proof}
    
    \section{Equigeodesics on Riemannian $M$-spaces}\label{applications}
    In this section, we will apply the previous results to investigate equigeodesic vectors on Riemannian $M$-spaces, which are a class of homogeneous spaces introduced by Wang in \cite{W}. For this purpose, we will follow the exposition of the isotropy representation of $M$-spaces as laid out by Arvanitoyeorgos, Wang, and Zhao in \cite{AWZ}.\\  
    
    Consider a generalized flag manifold $G/K$, where $G$ is a compact simple Lie group and $K$ is the centralizer of a torus $S$ in $G.$ Then $K$ is expressed as $K=S \times K_1,$ where $K_1$ is the semisimple part of $K.$ The associated $M$-space is the homogeneous space $G/K_1.$ Let us denote the Lie algebras of $G$ and $K$ by $\mathfrak{g}$ and $\mathfrak{k}$ and their complexifications by $\mathfrak{g}^{\mathbb{C}}$ and $\mathfrak{k}^{\mathbb{C}},$ respectively. Denote by $B$ the Killing form of $\mathfrak{g}$ and set $(\cdot,\cdot):=-B.$ Then $(\cdot,\cdot)$ is an $\Ad(G)$-invariant inner product on $\mathfrak{g}.$ We have a reductive decomposition $\mathfrak{g}=\mathfrak{k}\oplus\mathfrak{m},$ where $\mathfrak{m}$ is the orthogonal complement of $\mathfrak{k}$ in $\mathfrak{g}$ with respect to $(\cdot,\cdot).$ Let $T$ be a maximal torus of $G$ that contains $S,$ $\mathfrak{a}$ its Lie algebra and $\mathfrak{a}^\mathbb{C}$ the complexification of $\mathfrak{a}.$ Then $\mathfrak{a}^\mathbb{C}$ is a Cartan subalgebra of $\mathfrak{g}^\mathbb{C}$ and we have a decomposition of the form
    \begin{equation*}
        \mathfrak{g}^\mathbb{C}=\mathfrak{a}^\mathbb{C}\oplus\left(\sum\limits_{\alpha\in R}\mathfrak{g}_{\alpha}^\mathbb{C}\right),
    \end{equation*}
    where $R$ is a root system of $\mathfrak{g}^\mathbb{C}$ with respect to $\mathfrak{a}^\mathbb{C},$ and each $$\mathfrak{g}_\alpha^\mathbb{C}=\left\{X\in\mathfrak{g}^\mathbb{C}:\left[H,X\right]=\alpha(H)X,\ \text{for all}\ H\in\mathfrak{a}^\mathbb{C}\right\}$$ is the corresponding root space associated with $\alpha.$\\
    
    Since $\mathfrak{a}^\mathbb{C}\subseteq\mathfrak{k}^\mathbb{C},$ then $\mathfrak{a}^{\mathbb{C}}$ is also a Cartan subalgebra of $\mathfrak{k}^\mathbb{C}.$ Consequently, there exists a root system $R_K\subseteq R$ of $\mathfrak{k}^\mathbb{C}$ with respect to $\mathfrak{a}^\mathbb{C}$ such that $$\mathfrak{k}^\mathbb{C}=\mathfrak{a}^\mathbb{C}\oplus\left(\sum\limits_{\alpha\in R_K}\mathfrak{g}_{\alpha}^\mathbb{C}\right).$$
    Fix a set $R^+$ of positive roots and define $$R_K^+:=R^+\cap R_K,\ R_M:=R-R_K,\ \text{and}\ R_M^+:=R^+-R_K^+.$$ Choose a Weyl basis $\{E_\alpha,H_\alpha:\alpha\in R\}$ of $\mathfrak{g}^\mathbb{C}$ such that $H_\alpha=\left[E_\alpha,E_{-\alpha}\right]$ and $B(E_\alpha,E_{-\alpha})=-1$ for all $\alpha\in R.$ Then $$\mathfrak{k}=\left(\sum\limits_{\alpha\in R^+}\mathbb{R}\sqrt{-1}H_\alpha\right)\oplus\left(\sum\limits_{\alpha\in R_K^+}\left(\mathbb{R}A_\alpha\oplus\mathbb{R}S_\alpha\right)\right)$$ and $$\mathfrak{m}=\sum\limits_{\alpha\in R_M^+}\left(\mathbb{R}A_\alpha\oplus\mathbb{R}S_\alpha\right),$$ where $$A_\alpha:=E_{\alpha}-E_{-\alpha}\ \text{and}\  S_\alpha:=\sqrt{-1}\left(E_{\alpha}+E_{-\alpha}\right),\ \alpha\in R.$$ Let $\Pi=\{\alpha_1,...,\alpha_r,\phi_1,...,\phi_k\}$ be a system of simple roots of $R$ such that $\Pi_K:=\{\phi_1,...,\phi_k\}$ is a system of simple roots of $R_K$ and let $\{\Lambda_1,...,\Lambda_r,\tilde{\Lambda}_1,...,\tilde{\Lambda}_{k}\}$ be the set of fundamental weights corresponding to $\Pi.$ Consider the subspace $$\mathfrak{t}=\{X\in\sqrt{-1}\mathfrak{a}:\phi(X)=0,\ \text{for all}\ \phi\in R_K\}\subseteq\mathfrak{a}^\mathbb{C},$$ and the restriction map $$\kappa:\left(\mathfrak{a}^\mathbb{C}\right)^*\rightarrow\mathfrak{t}^*;\ \alpha\mapsto \alpha\big{|}_{\mathfrak{t}}.$$ The elements of $R_\mathfrak{t}:=\kappa(R_M)$ and $R^+_\mathfrak{t}:=\kappa(R_M^+)$ are called $\mathfrak{t}$-roots and positive $\mathfrak{t}$-roots respectively. If $R_\mathfrak{t}^+=\{\xi_1,...,\xi_s\},$ and we set \begin{equation*}\label{m:i}
    \mathfrak{m}_i:=\sum\limits_{\begin{subarray}{c}\alpha\in R^+_M\\\kappa(\alpha)=\xi_i\end{subarray}}\left(\mathbb{R}A_\alpha\oplus\mathbb{R}S_\alpha\right),\ i=1,...,s,
    \end{equation*}
    then each $\mathfrak{m}_i$ is an $\Ad(K)$-invariant irreducible subspace of $\mathfrak{m}$ and $$\mathfrak{m}=\mathfrak{m}_1\oplus\cdots\oplus\mathfrak{m}_s,$$ where $\mathfrak{m}_1,...,\mathfrak{m}_s$ are pairwise $(\cdot,\cdot)$-orthogonal and $\mathfrak{m}_i\ncong\mathfrak{m}_j,\ \text{for}\ i\neq j.$\\

    For the corresponding $M$-space $G/K_1,$ a reductive $(\cdot,\cdot)$-orthogonal decomposition of $\mathfrak{g}$ is given by $$\mathfrak{g}=\mathfrak{k_1}\oplus(\mathfrak{s}\oplus\mathfrak{m}),$$ where $\mathfrak{k}_1$ and $\mathfrak{s}$ are the the Lie algebras of $K_1$ and $S$ respectively. Consequently, $\mathfrak{n}:=\mathfrak{s}\oplus\mathfrak{m}$ is identified with the tangent space $T_o(G/K_1)$ at $o=eK_1.$ Since $K_1\subseteq K,$ then each $\mathfrak{m}_i$ is an $\Ad(K_1)$-invariant (not necessarily irreducible) subspace of $\mathfrak{n}.$  On the other hand, the adjoint representation of $K_1$ on $\mathfrak{s}$ is trivial, given that $K_1\subseteq K$ and $K$ centralizes $S.$ The Lie algebra $\mathfrak{s}$ can be written as
    $$\mathfrak{s}=\mathbb{R}\sqrt{-1}h_{\Lambda_1}\oplus\cdots\oplus \mathbb{R}\sqrt{-1}h_{\Lambda_r},$$ where for each $\alpha\in\left(\mathfrak{a}^\mathbb{C}\right)^*,$ $h_\alpha$ is defined implicitly by the formula $(H,h_\alpha)=\alpha(H).$

    \begin{theorem}[\cite{AWZ}]\label{AWZ:theorem} Let $\mathfrak{m}_i$ be an $\Ad(K)$-invariant irreducible subspace of $\mathfrak{m}$ and suppose that $\mathfrak{m}_i$ is $\Ad(K_1)$-reducible. Then $\mathfrak{m}_i$ is decomposed as $\mathfrak{m}_i=\mathfrak{n}_1^{i}\oplus\mathfrak{n}_2^{i},$ where $\mathfrak{n}_1^{i}$ and $\mathfrak{n}_2^{i}$ are $\Ad(K_1)$-invariant irreducible subspaces.
    \end{theorem}

    As a consequence of Theorem \ref{AWZ:theorem}, we have that the decomposition of $\mathfrak{n}$ into $\Ad(K_1)$-irreducible submodules has the form \begin{equation*}
        \mathfrak{n}=\mathbb{R}\sqrt{-1}h_{\Lambda_1}\oplus\cdots\oplus \mathbb{R}\sqrt{-1}h_{\Lambda_r}\oplus(\mathfrak{n}_1^1\oplus\mathfrak{n}_2^1)\oplus\cdots\oplus(\mathfrak{n}_1^{r'}\oplus\mathfrak{n}_2^{r'})\oplus\mathfrak{m}_{r'+1}\oplus\cdots\oplus\mathfrak{m}_s,
    \end{equation*}
    where $\mathbb{R}\sqrt{-1}h_{\Lambda_i}\cong\mathbb{R}\sqrt{-1}h_{\Lambda_j},\ i,j=1,...,r$ and $\mathfrak{n}_1^{i}\cong\mathfrak{n}_2^{i},\ i=1,...,r'.$
    \begin{remark} It is possible that each $\Ad(K)$-irreducible submodule is also $\Ad(K_1)$-irreducible. Consequently, the terms $(\mathfrak{n}_1^i\oplus\mathfrak{n}_2^{i})$ will not appear in the decomposition above. In such situations, we can set $r'=0.$ 
        
    \end{remark}
    \begin{lemma}\label{lemma:4:3} Let $X=X_{\mathfrak{s}}+X_{\mathfrak{m}}\in\mathfrak{n}$  \textnormal{(}$X_{\mathfrak{s}}\in\mathfrak{s},$ $X_{\mathfrak{m}}\in\mathfrak{m}$\textnormal{)} be an equigeodesic vector on the Riemannian $M$-space $G/K_1.$ Then either $X_{\mathfrak{m}}=0,$ or $X_{\mathfrak{s}}=0$ and $X_{\mathfrak{m}}$ is an equigeodesic vector on the generalized flag manifold $G/K.$ 
    \end{lemma}
    \begin{proof} Since $(\cdot,\cdot)$ is $\Ad(K)$-invariant, then it is also $\Ad(K_1)$-invariant. Therefore, we may assume that $\mathfrak{n}_1^{i}$ and $\mathfrak{n}_2^{i}$ are $(\cdot,\cdot)$-orthogonal. Let $\{Z_1,...,Z_r\}$ be an $(\cdot,\cdot)$-orthonormal basis of $\mathfrak{s}.$ Since the adjoint representation of $K_1$ on $\mathfrak{s}$ is trivial, then the subspaces $\mathbb{R}Z_i,\ i=1,...,r$ are  $\Ad(K_1)$-invariant irreducible and pairwise $(\cdot,\cdot)$-orthogonal subspaces of $\mathfrak{n}.$ Let \begin{align}\label{projections}
    \begin{split}
    &P_i:\mathfrak{n}\rightarrow\mathbb{R}Z_i,\ i=1,...,r,\\ &P_{ij}:\mathfrak{n}\rightarrow\mathfrak{n}_j^{i},\ i=1,..,r',\ j=1,2,\ \text{and}\\ &\tilde{P}_{i}:\mathfrak{n}\rightarrow \mathfrak{m}_i\ i=1,...,s
    \end{split}
    \end{align}
    be the orthogonal projections onto $\mathbb{R}Z_i,$  $\mathfrak{n}_j^{i},$ and $\mathfrak{m}_i,$ respectively. For each $i\in\{1,...,r'\},$ let $T_{i1}:\mathfrak{n}_1^i\rightarrow\mathfrak{n}_{2}^{i}$ be an intertwining operator preserving $(\cdot,\cdot)$ and let $T_{i2}:=T_{i1}^{-1}.$ Additionally, define $T_{i}^{j}:\mathbb{R}Z_i\rightarrow \mathbb{R}Z_j;\ cZ_i\mapsto cZ_j,\ i,j=1,...,r$ so that $T_i^{j}$ is also an interwining operator preserving $(\cdot,\cdot).$ If $X=X_{\mathfrak{s}}+X_{\mathfrak{m}}\in\mathfrak{n}$ is an equigeodesic vector, then, by Theorem \ref{main:2}, we have
    \begin{align}\label{main:formula:M:spaces}
    \begin{split}    &\left[X_{\mathfrak{s}}+X_{\mathfrak{m}},T_i^{j}(P_i(X))+T_{j}^{i}(P_j(X))\right]_{\mathfrak{n}}=0,\ i,j=1,...,r,\\
        &\left[X_{\mathfrak{s}}+X_{\mathfrak{m}},P_{ij}(X)\right]_{\mathfrak{n}}=0,\ i=1,...,r',\ j=1,2,\\
        &\left[X_{\mathfrak{s}}+X_{\mathfrak{m}},T_{i1}(P_{i1}(X))+T_{i2}(P_{i2}(X))\right]_{\mathfrak{n}}=0,\ i=1,...,r'\ \text{and}\\
        &[X_{\mathfrak{s}}+X_{\mathfrak{m}},\tilde{P}_i(X)]_{\mathfrak{n}}=0,\ i=r'+1,...,s.
    \end{split}
    \end{align}
    Let $X_{\mathfrak{s}}=\sum\limits_{k=1}^rc_kZ_k,$ then $P_k(X)=c_kZ_k,\ k=1,...,r$ and
    \begin{align*} 0=&\left[X_{\mathfrak{s}}+X_{\mathfrak{m}},T_i^{j}(P_i(X))+T_{j}^{i}(P_j(X))\right]_{\mathfrak{n}}\\
    =& \left[X_{\mathfrak{s}}+X_{\mathfrak{m}},T_i^j(c_iZ_i)+T_j^{i}(c_jZ_j)\right]_{\mathfrak{n}}\\
    =&\left[X_{\mathfrak{s}}+X_{\mathfrak{m}},c_iZ_j+c_jZ_i\right]_{\mathfrak{n}}\\
    =&\left[X_{\mathfrak{s}},c_iZ_j+c_jZ_i\right]_{\mathfrak{n}}+\left[X_{\mathfrak{m}},c_iZ_j+c_jZ_i\right]_{\mathfrak{n}}\\
    =&c_i\left[X_{\mathfrak{m}},Z_j\right]_{\mathfrak{n}}+c_j\left[X_{\mathfrak{m}},Z_i\right]_{\mathfrak{n}},\ i,j=1,...,r,
    \end{align*}
    where the last equality holds because $\mathfrak{s}$ is abelian and $X_{\mathfrak{s}}$ and $c_iZ_j+c_jZ_i$ are in $\mathfrak{s}.$ Observe that for $i=j$ we have that $$c_i\left[X_{\mathfrak{m}},Z_i\right]_{\mathfrak{n}}=0,\ i=1,...,r,$$ which  implies $$c_i^2\left[X_{\mathfrak{m}},Z_i\right]_{\mathfrak{n}}=0,\ i=1,...,r.$$ On the other hand, if $i\neq j,$ then 
    \begin{align*}
        c_i\left[X_{\mathfrak{m}},Z_j\right]_{\mathfrak{n}}+c_j\left[X_{\mathfrak{m}},Z_i\right]_{\mathfrak{n}}=0\Longrightarrow &\ c_i\left\{c_j\left[X_{\mathfrak{m}},Z_j\right]_{\mathfrak{n}}\right\}+c_j^2\left[X_{\mathfrak{m}},Z_i\right]_{\mathfrak{n}}=0\\
        \Longrightarrow &\ 0+c_j^2\left[X_{\mathfrak{m}},Z_i\right]_{\mathfrak{n}}=0\\
        \Longrightarrow &\ c_j^2\left[X_{\mathfrak{m}},Z_i\right]_{\mathfrak{n}}=0.,
    \end{align*}
    so we obtain $c_j^2\left[X_{\mathfrak{m}},Z_i\right]_{\mathfrak{n}}=0$ for all $i,j\in\{1,...,r\}.$ By adding this equalities in $j$ we have $$\left(c_1^2+\cdots+c_r^2\right)\left[X_{\mathfrak{m}},Z_i\right]_{\mathfrak{n}}=0.$$
    If $X_{\mathfrak{s}}\neq 0,$ then $c_1^2+\cdots+c_r^2\neq 0.$ Thus, $\left[X_{\mathfrak{m}},Z_i\right]_{\mathfrak{n}}=0$ and this equality holds for all $i\in\{1,...,r\}.$ Now, observe that $Z_i\in\mathfrak{s}\subseteq\mathfrak{k},$ so $\left[X_{\mathfrak{m}},Z_i\right]\in\left[\mathfrak{m},\mathfrak{k}\right]\subseteq\mathfrak{m},$ which implies that $\left[X_{\mathfrak{m}},Z_i\right]_{\mathfrak{k}_1}=0.$ Hence,
    \begin{align*}
        \left[X_{\mathfrak{m}},Z_i\right]=\left[X_{\mathfrak{m}},Z_i\right]_{\mathfrak{k}_1}+\left[X_{\mathfrak{m}},Z_i\right]_{\mathfrak{n}}=0+0=0,\ i=1,...,r.
    \end{align*} Since $\{Z_1,...,Z_r\}$ is a basis of $\mathfrak{s},$ we conclude that $X_{\mathfrak{m}}$ is in the centralizer of $\mathfrak{s}$ in $\mathfrak{g},$ that is, $X_{\mathfrak{m}}\in\mathfrak{k}.$ But $X_{\mathfrak{m}}$ is also in $\mathfrak{m},$ so $X_{\mathfrak{m}}\in\mathfrak{k}\cap\mathfrak{m}=\{0\},$ i.e., $X_{\mathfrak{m}}=0.$ So far, we have established that if $X\in\mathfrak{n}$ is an equigeodesic vector, then $X_{\mathfrak{s}}=0$ or $X_{\mathfrak{m}}=0.$ Now, let us assume that $X\in\mathfrak{n}$ is equigeodesic and that $X_{\mathfrak{s}}=0.$ Then, by \eqref{main:formula:M:spaces}, we can conclude that $$\left[X_{\mathfrak{m}},\tilde{P}_i(X_{\mathfrak{m}})\right]_{\mathfrak{n}}=0,\ i=r'+1,...,s,$$ and $$\ \left[X_{\mathfrak{m}},P_{ij}(X_{\mathfrak{m}})\right]_{\mathfrak{n}}=0,\ i=1,...,r',\ j=1,2.$$ 
    
    Now, note that $P_{i1}+P_{i2}=\tilde{P}_i,\ i=1,...,r'.$ Thus,  we have $$\left[X_{\mathfrak{m}},\tilde{P}_i(X_{\mathfrak{m}})\right]_{\mathfrak{n}}=0,\ i=1,...,s.$$ Since, $\mathfrak{m}\subseteq \mathfrak{n},$ this implies $$\left[X_{\mathfrak{m}},\tilde{P}_i(X_{\mathfrak{m}})\right]_{\mathfrak{n}}=0,\ i=1,...,s\Longrightarrow \left[X_{\mathfrak{m}},\tilde{P}_i(X_{\mathfrak{m}})\right]_{\mathfrak{m}}=0,\ i=1,...,s.$$ Due to Theorem \ref{main:3}, this indicates that $X_{\mathfrak{m}}$ is equigeodesic on the generalized flag manifold $G/K$ because, within the isotropy representation of $G/K,$ there are not equivalent summands.
    \end{proof}
    \begin{theorem}\label{equigeodesic:M:spaces}
        Let $G/K_1$ be the Riemannian $M$-space corresponding to the generalized flag manifold $G/K,$ let $(\cdot,\cdot)$ be the negative of the Killing form on $\mathfrak{g},$ and let $$\mathfrak{g}=\mathfrak{k}_1\oplus\mathfrak{n}=\mathfrak{k}_1\oplus\mathfrak{s}\oplus\mathfrak{m}$$ be the corresponding $(\cdot,\cdot)$-orthogonal reductive decomposition of $\mathfrak{g}.$\\ 
        
        \begin{itemize}
            \item[$(1)$] If each $\Ad(K)$-irreducible submodule of $\mathfrak{m}$ is also $\Ad(K_1)$-irreducible, then $X\in\mathfrak{n}$ is equigeodesic on $G/K_1$ if and only if either $X\in\mathfrak{s},$ or $X\in\mathfrak{m}$ and $X$ is equigeodesic on $G/K.$\\

            \item[$(2)$] If $\dim\mathfrak{m}_j>2$ for each $\Ad(K)$-irreducible submodule $\mathfrak{m}_j$ which is also $\Ad(K_1)$-irreducible, then $X\in\mathfrak{n}$ is equigeodesic on $G/K_1$ if and only if either $X\in\mathfrak{s},$ or $X\in\mathfrak{m}$ and $X$ is equigeodesic on $G/K$ and either $X_{\mathfrak{m}_i}=0$ or  $\left[X,\mathfrak{m}_i\right]_{\mathfrak{n}}=0,$ for each $\Ad(K)$-irreducible submodule $\mathfrak{m}_i$ with $\dim\mathfrak{m}_i=2.$
        \end{itemize}
    \end{theorem}
    \begin{proof} We will use the notations introduced in the proof of Lemma \ref{lemma:4:3}. Let us prove $(1).$ If $X\in\mathfrak{n}$ is equigeodesic on $G/K_1,$ then, by Lemma \ref{lemma:4:3}, $X\in\mathfrak{s},$ or $X\in\mathfrak{m}$  and $X$ is equigeodesic on $G/K.$ For the reverse implication, observe that since each $\Ad(K)$-irreducible submodule of $\mathfrak{m}$ is also $\Ad(K_1)$-irreducible, then the only isotropy summands of $\mathfrak{n}$ with multiplicity greater than one are $$\mathbb{R}Z_i,\ i=1,...,r,$$ each of them having dimension one. Consequently, $\textnormal{End}(\mathbb{R}Z_i)\cong \mathbb{R},\ i=1,...,r,$ and hypothesis of Theorem \ref{main:3} are satisfied. This means that a vector $X=X_{\mathfrak{s}}+X_{\mathfrak{m}}\in\mathfrak{n}$ is equigeodesic on $G/K_1$ if and only if
    \begin{align*}
    \begin{split}    &\left[X_{\mathfrak{s}}+X_{\mathfrak{m}},T_i^{j}(P_i(X))+T_{j}^{i}(P_j(X))\right]_{\mathfrak{n}}=0,\ i,j=1,...,r,\\
    &[X_{\mathfrak{s}}+X_{\mathfrak{m}},\tilde{P}_i(X)]_{\mathfrak{n}}=0,\ i=1,...,s.
    \end{split}
    \end{align*} 
    If $X\in\mathfrak{s}$ then $X_{\mathfrak{m}}=\tilde{P}_i(X)=0,\ i=1,...,s.$ Thus, $X$ is equigeodesic if and only if $$\left[X,T_i^{j}(P_i(X))+T_{j}^{i}(P_j(X))\right]_{\mathfrak{n}}=0,\ i,j=1,...,r$$ which holds true because $\left[X,T_i^{j}(P_i(X))+T_{j}^{i}(P_j(X))\right]\in\left[\mathfrak{s},\mathfrak{s}\right]=\{0\},\ i,j=1,...,r.$ On the other hand, if $X\in\mathfrak{m}$ then $X_{\mathfrak{s}}=P_i(X)=0,\ i=1,...,r,$ and it is equigeodesic on $G/K_1$ if and only if $$\left[X,\tilde{P}_i(X)\right]_{\mathfrak{n}}=0,\ i=1,...,s.$$ Suppose that $X\in\mathfrak{m}$ and that $X$ is equigeodesic on $G/K,$ then $$\left[X,\tilde{P}_i(X)\right]_{\mathfrak{m}}=0,\ i=1,...,s.$$ Since $$\left[X,\tilde{P}_i(X)\right]_{\mathfrak{n}}=\left[X,\tilde{P}_i(X)\right]_{\mathfrak{s}}+\left[X,\tilde{P}_i(X)\right]_{\mathfrak{m}},$$ then we only have to prove that $$\left[X,\tilde{P}_i(X)\right]_{\mathfrak{s}}=0,\ i=1,...,s.$$ Given that $$\tilde{P}_i(X)\in\mathfrak{m}_i=\bigoplus\limits_{\begin{subarray}{c}\alpha\in R^+_M\\\kappa(\alpha)=\xi_i\end{subarray}}\left(\mathbb{R}A_\alpha\oplus\mathbb{R}S_\alpha\right),$$ then there exists a family $\{x_\alpha,\ y_\alpha:\alpha\in R_M^+\ \text{and}\  \kappa(\alpha)=\xi_i\}\subseteq\mathbb{R}$ such that $$\tilde{P}_i(X)=\sum\limits_{\begin{subarray}{c}\alpha\in R^+_M\\\kappa(\alpha)=\xi_i\end{subarray}}\left(x_\alpha A_\alpha+y_\alpha S_\alpha\right),\ i=1,...,s.$$ Therefore, we have
        \begin{align}\label{P:i:P:j}
        \begin{split}    \left[\tilde{P}_i(X),\tilde{P}_j(X)\right] = &\left[\sum\limits_{\begin{subarray}{c}\alpha\in R^+_M\\\kappa(\alpha)=\xi_i\end{subarray}}\left(x_\alpha A_\alpha+y_\alpha S_\alpha\right),\sum\limits_{\begin{subarray}{c}\beta\in R^+_M\\\kappa(\beta)=\xi_j\end{subarray}}\left(x_\beta A_\beta+y_\beta S_\beta\right)\right]\\
            \\
            =&\sum\limits_{\begin{subarray}{c}\alpha\in R^+_M\\\kappa(\alpha)=\xi_i\end{subarray}}\sum\limits_{\begin{subarray}{c}\beta\in R^+_M\\\kappa(\beta)=\xi_j\end{subarray}}\left(x_\alpha x_\beta\left[A_\alpha,A_\beta\right]+x_\alpha y_\beta\left[A_\alpha,S_\beta\right]\right)\\
            \\
            &+\sum\limits_{\begin{subarray}{c}\alpha\in R^+_M\\\kappa(\alpha)=\xi_i\end{subarray}}\sum\limits_{\begin{subarray}{c}\beta\in R^+_M\\\kappa(\beta)=\xi_j\end{subarray}}\left(x_\beta y_\alpha\left[A_\alpha,S_\beta\right]+y_\alpha y_\beta\left[S_\alpha,S_\beta\right]\right).
        \end{split}
        \end{align}
    The elements $A_\alpha,\ S_\alpha,\ \sqrt{-1}H_\alpha$ satisfy the Lie bracket relations:
    
    $$\begin{array}{ll}
        \left[\sqrt{-1}H_{\alpha},A_\beta\right]=\beta(H_\alpha)S_\beta, & \left[A_\alpha,A_\beta\right]=m_{\alpha,\beta}A_{\alpha+\beta}+m_{-\alpha,\beta}A_{\alpha-\beta},\ \alpha\neq\beta, \\
        \\
        \left[\sqrt{-1}H_{\alpha},S_\beta\right]=-\beta(H_\alpha)A_\beta, & \left[S_\alpha,S_\beta\right]=-m_{\alpha,\beta}A_{\alpha+\beta}-m_{\alpha,-\beta}A_{\alpha-\beta},\ \alpha\neq\beta,\\
        \\
        \left[A_\alpha,S_\alpha\right]=2\sqrt{-1}H_\alpha, & \left[A_\alpha,S_\beta\right]=m_{\alpha,\beta}S_{\alpha+\beta}+m_{\alpha,-\beta}S_{\alpha-\beta},\ \alpha\neq\beta,
        \end{array}$$
        
        for some real constants $m_{\alpha,\beta},\ \alpha,\beta\in R$ such that  $m_{\alpha,\beta}=0$ if $\alpha+\beta$ is not a root, $-m_{\alpha,\beta}=m_{-\alpha,-\beta},$ and $m_{\beta,\alpha}=-m_{\alpha,\beta}.$ Since $\mathfrak{s}$ is spanned by elements of the form $\sqrt{-1}H_\gamma$ ($\gamma\in R),$ then, according to \eqref{P:i:P:j}, $$\left[\tilde{P}_j(X),\tilde{P}_i(X)\right]_{\mathfrak{s}}=0$$ unless $$\{\alpha\in R_M^+:\kappa(\alpha)=\xi_i\}\cap\{\beta\in R_M:\kappa(\beta)=\xi_j\}\neq\emptyset,$$ that is, $\left[\tilde{P}_j(X),\tilde{P}_i(X)\right]_{\mathfrak{s}}= 0$ unless $i=j.$ Obviously, $\left[\tilde{P}_i(X),\tilde{P}_i(X)\right]_{\mathfrak{s}}= 0,$ hence
        \begin{align*}
            \left[X,\tilde{P}_i(X)\right]_{\mathfrak{s}}=\left[\sum\limits_{j=1}^s\tilde{P}_j(X),\tilde{P}_i(X)\right]_{\mathfrak{s}}=\sum\limits_{j=1}^s\left[\tilde{P}_j(X),\tilde{P}_i(X)\right]_{\mathfrak{s}}=0.
        \end{align*}
        This completes the proof of $(1).$\\
        
        Let us prove $(2)$. If $\mathfrak{m}_i$ is a two-dimensional $\Ad(K)$-irreducible submodule of $\mathfrak{m},$ then $\mathfrak{m}_i$ is $\Ad(K_1)$-reducible and we can set $$\mathfrak{n}_1^{i}=\mathbb{R}A_{\gamma_i}\ \text{and}\ \mathfrak{n}_2^{i}=\mathbb{R}S_{\gamma_i},$$ for some $\gamma_i\in R_M^+$ such that $\kappa(\gamma_i)=\xi_i.$ Moreover, $$T_{i1}:A_{\gamma_i}\mapsto S_{\gamma_i},\ \text{and}\ T_{i2}:S_{\gamma_i}\mapsto A_{\gamma_i}$$ are intertwining operators preserving $(\cdot,\cdot)$ such that $\left(T_{i1}\right)^{-1}=T_{i2}$ (see \cite[Lemma 2 and Remark 2]{AWZ}).\\
        
        Let $X\in\mathfrak{n}$ be an equigeodesic vector on $G/K_1$. By Lemma \ref{lemma:4:3}, we already know that either $X\in\mathfrak{s},$ or $X\in\mathfrak{m}$ and $X$ is equigeodesic on  $G/K.$ Suppose that $X\in\mathfrak{m}$ and $X$ is equigeodesic on $G/K.$ Let us write 
        $$P_{i1}(X)=x_{\gamma_i} A_{\gamma_i},\ P_{i2}(X)=y_{\gamma_i} S_{\gamma_i},\ i=1,...,r'.$$ 
        (see \eqref{projections}). 
        The equations \eqref{main:formula} for $X$ can be rewritten as follows:
        \begin{align*}
            &\left[X,x_{\gamma_i}A_{\gamma_i}\right]_{\mathfrak{n}}=\left[X,y_{\gamma_i}S_{\gamma_i}\right]_{\mathfrak{n}}=0,\ i=1,...,r',\\
            &\left[X,x_{\gamma_i}S_{\gamma_i}+y_{\gamma_i}A_{\gamma_i}\right]_{\mathfrak{n}}=0,\ i=1,...,r'.\\
            &\left[X,\tilde{P}_i(X)\right]_{\mathfrak{n}}=0,\ i=r'+1,...,s.
        \end{align*}
        In particular, we have
        \begin{align*}
            &x_{\gamma_i}\left[X,A_{\gamma_i}\right]_{\mathfrak{n}}=y_{\gamma_i}\left[X,S_{\gamma_i}\right]_{\mathfrak{n}}=0,\\
            &x_{\gamma_i}\left[X,S_{\gamma_i}\right]_{\mathfrak{n}}+y_{\gamma_i}\left[X,A_{\gamma_i}\right]_{\mathfrak{n}}=0,\ i=1,...,r'.
        \end{align*}
        But
        \begin{align*}
            x_{\gamma_i}\left[X,A_{\gamma_i}\right]_{\mathfrak{n}}=0\Longrightarrow\ & x_{\gamma_i}^2\left[X,A_{\gamma_i}\right]_{\mathfrak{n}}=0,\\
            y_{\gamma_i}\left[X,S_{\gamma_i}\right]_{\mathfrak{n}}=0\Longrightarrow\ & y_{\gamma_i}^2\left[X,S_{\gamma_i}\right]_{\mathfrak{n}}=0,
        \end{align*}
        and
        \begin{align*}
            x_{\gamma_i}\left[X,S_{\gamma_i}\right]_{\mathfrak{n}}+y_{\gamma_i}\left[X,A_{\gamma_i}\right]_{\mathfrak{n}}=0\Longrightarrow\ &x_{\gamma_i}^2\left[X,S_{\gamma_i}\right]_{\mathfrak{n}}+y_{\gamma_i}x_{\gamma_i}\left[X,A_{\gamma_i}\right]_{\mathfrak{n}}=0\\
            \Longrightarrow\ &x_{\gamma_i}^2\left[X,S_{\gamma_i}\right]_{\mathfrak{n}}=0,\\
            x_{\gamma_i}\left[X,S_{\gamma_i}\right]_{\mathfrak{n}}+y_{\gamma_i}\left[X,A_{\gamma_i}\right]_{\mathfrak{n}}=0\Longrightarrow\ &x_{\gamma_i}y_{\gamma_i}\left[X,S_{\gamma_i}\right]_{\mathfrak{n}}+y_{\gamma_i}^2\left[X,A_{\gamma_i}\right]_{\mathfrak{n}}=0\\
            \Longrightarrow\ &y_{\gamma_i}^2\left[X,A_{\gamma_i}\right]_{\mathfrak{n}}=0,\\
        \end{align*}
        hence
        \begin{align*}
            &\left(x_{\gamma_i}^2+y_{\gamma_i}^2\right)\left[X,A_{\gamma_i}\right]_{\mathfrak{n}}=x_{\gamma_i}^2\left[X,A_{\gamma_i}\right]_{\mathfrak{n}}+y_{\gamma_i}^2\left[X,A_{\gamma_i}\right]_{\mathfrak{n}}=0\\
            &\left(x_{\gamma_i}^2+y_{\gamma_i}^2\right)\left[X,S_{\gamma_i}\right]_{\mathfrak{n}}=x_{\gamma_i}^2\left[X,S_{\gamma_i}\right]_{\mathfrak{n}}+y_{\gamma_i}^2\left[X,S_{\gamma_i}\right]_{\mathfrak{n}}=0.\\
        \end{align*}
        From the above, we conclude that either $x_{\gamma_i}^2+y_{\gamma_i}^2=0,$ in which case $X_{\mathfrak{m}_i}=0,$ or $\left[X,A_{\gamma_i}\right]_{\mathfrak{n}}=\left[X,S_{\gamma_i}\right]_{\mathfrak{n}}=0,$ in which case $\left[X,\mathfrak{m}_i\right]_{\mathfrak{n}}=0.$ Summarizing, we proved that if $X\in\mathfrak{n}$ is equigeodesic on $G/K_1$ then either $X\in\mathfrak{s},$ or $X\in\mathfrak{m},$ $X$ is equigeodesic on $G/K$ and either $X_{\mathfrak{m}_i}=0$ or $\left[X,\mathfrak{m}_i\right]_{\mathfrak{n}}=0,$ for each $\Ad(K)$-irreducible submodule $\mathfrak{m}_i$ of dimension two.
        
        For the reverse implication, we can proceed as in the proof of the reverse implication of $(1).$ This completes the proof of $(2).$
        \end{proof}
        \begin{remark} Theorem \ref{equigeodesic:M:spaces} provides a characterization of equigeodesic vectors within certain class of $M$-spaces. It is important to note that this characterization depends on determining whether a vector is equigeodesic on the corresponding generalized flag manifold. For a comprehensive understanding of equigeodesic vectors on generalized flag manifolds, we refer to \cite{CGN,S2,WZ,X}.
        \end{remark}

\section*{Acknowledgements}
Brian Grajales is supported by grant 2023/04083-0 (São Paulo Research Foundation FAPESP). Lino Grama is partially supported by FAPESP grant 2018/13481-0.

\bibliographystyle{apa}

\end{document}